\documentclass[11pt, a4paper,leqno]{amsart}
\usepackage{amsmath,amsthm,amscd,amssymb,amsfonts, amsbsy}
\usepackage{latexsym}
\usepackage{txfonts}
\usepackage{exscale}

\usepackage[colorlinks,citecolor=red,pagebackref,hypertexnames=false]{hyperref}
\usepackage{color}

\newcommand{\Rd}{\color{red}}

\calclayout
\allowdisplaybreaks

\theoremstyle{plain}
\newtheorem{theorem}[equation]{Theorem}
\newtheorem{lemma}[equation]{Lemma}
\newtheorem{corollary}[equation]{Corollary}
\newtheorem{proposition}[equation]{Proposition}

\newtheorem{observation}[equation]{Observation}
\theoremstyle{definition}
\newtheorem{definition}[equation]{Definition}

\theoremstyle{remark}
\newtheorem{remark}[equation]{Remark}

\numberwithin{equation}{section}

\newcommand{\RR}{{\mathbb{R}}}
\newcommand{\eps}{\varepsilon}

\newcommand{\dist}{\operatorname{dist}}

\newcommand{\Reg}{\operatorname{Reg}}
\newcommand{\GLem}{\operatorname{GLem}}
\newcommand{\WGLem}{\operatorname{WGLem}}
\newcommand{\BWGLem}{\operatorname{BWGLem}}
\newcommand{\BP}{\operatorname{BP}}
\newcommand{\BPLG}{\operatorname{BPLG}}
\newcommand{\GPG}{\operatorname{GPG}}
\newcommand{\LG}{\operatorname{LG}}

\newcommand{\re}{\mathbb{R}}
\newcommand{\rn}{\mathbb{R}^n}

\newcommand{\ree}{\mathbb{R}^{n+1}}

\newcommand{\dd}{\mathbb{D}}

\newcommand{\C}{\mathcal{C}}

\newcommand{\F}{\mathcal{F}}

\newcommand{\M}{\mathcal{M}}

\newcommand{\B}{\mathcal{B}}

\newcommand{\sbf}{{\bf S}}

\newcommand{\G}{\mathcal{G}}

\newcommand{\mut}{\mathfrak{m}}

\renewcommand{\emptyset}{\mbox{\textup{\O}}}

\DeclareMathOperator{\diam}{diam}

\begin{document}
\allowdisplaybreaks

\title[Corona implies $\BP^2$]{Coronizations and big pieces in metric spaces}

\author{S. Bortz}
\address{Department of Mathematics
\\
University of Alabama
\\ 
Tuscaloosa, AL, 35487, USA}
\email{sbortz@ua.edu}
\author{J. Hoffman}
\address{Department of Mathematics
\\
University of Missouri
\\
Columbia, MO 65211, USA}
\email{jlh82b@mail.missouri.edu}
\author{S. Hofmann}
\address{
Department of Mathematics
\\
University of Missouri
\\
Columbia, MO 65211, USA} 
\email{hofmanns@missouri.edu}
\author{J-L. Luna Garcia}
\address{Department of Mathematics
\\
University of Missouri
\\
Columbia, MO 65211, USA}
\email{jlwwc@mail.missouri.edu}
\author{K. Nystr\"om}
\address{Department of Mathematics, Uppsala University, S-751 06 Uppsala, Sweden}
\email{kaj.nystrom@math.uu.se}

\thanks{The authors J. H., S. H., and J-L. L. G.  were partially supported by NSF grants  
DMS-1664047 and DMS-2000048. K.N. was partially supported by grant  2017-03805 from the Swedish research council (VR)}

\date{\today}


\begin{abstract}
We prove that coronizations with respect to arbitrary
$d$-regular sets (not necessarily graphs) imply big pieces squared of these (approximating) sets. This is
known (and due to  David and Semmes in the case of sufficiently large co-dimension,
and to Azzam and Schul in general) in the (classical) setting of Euclidean spaces with
Hausdorff measure of integer dimension,  where the approximating sets are Lipschitz graphs.
Our result is a far reaching generalization of these results and we prove that  coronizations imply big pieces squared is a generic property. In particular, our result applies, when suitably
interpreted, in metric spaces having  a fixed positive (perhaps non-integer) dimension, equipped with a Borel regular measure and with arbitrary approximating sets. As a novel application we highlight how to utilize  this general setting in the context of parabolic uniform rectifiability.
\end{abstract}

\maketitle

\tableofcontents

\section{Introduction}
The monumental works of G. David and S. Semmes \cite{DS1}, \cite{DS2} concerning equivalent characterization of uniformly rectifiable (UR) sets $E\subset\mathbb R^n$ remain a source of continuous inspiration for anyone interested in geometry and analysis. Their results apply in the Euclidean metric space
$(X,\dist,\mu) = (\rn,|\cdot|_2,H^d)$, where $d < n$ and $H^d$ is the $d$-dimensional  Hausdorff measure.
The following three characterizations  of uniformly rectifiability of an Ahlfors-regular set $E$ are proved in \cite{DS1} (we refer to \cite{DS1} for definitions and
precise statements):
\begin{itemize}
    \item $E$ admits a coronization with respect to Lipschitz graphs.
    \item $E$ has big pieces of bi-Lipschitz images.
    \item $E$ satisfies a `geometric lemma' quantified in terms of $\beta$-numbers.
\end{itemize}
In particular, in \cite{DS1}  it is proved that  uniformly rectifiability of an Ahlfors-regular set $E$ can be
characterized by the property that $E$ admits, for each
$\eta>0$, a corona decomposition with respect to
Lipschitz graphs in the class $\mathcal{E}=\mathcal{E}^{Lip}_\eta$ where $\mathcal{E}^{Lip}_\eta$ denotes the class of Lipschitz graphs with Lipschitz
 constant no larger
than  $\eta>0$.  The phrase that  `$E$ has big pieces of $\mathcal{E}$' means that $E$ has a
uniformly  `large amount' of coincidence with a set from $\mathcal{E}$,
at every location (point on $E$) and at every scale. This big pieces functor can be iterated (see Definition \ref{bigpiecesdef.def}) and it is natural, in light of \cite{CDM}, to ask if there is a $j \in \mathbb{N}$ such that every UR set is $\BP^j(\LG)$ where this notation means iterating the big pieces $(\BP)$ functor $j$ times and $\LG$ is the collection of Lipschitz graphs (with uniform control on the Lipschitz constants). In fact, while it was proved  by Hrycak (unpublished), that not every UR set is $\BPLG=\BP^1(\LG)$, in  \cite{DS2} it is proved that UR sets are $\BP^2(\LG)$, when $n \ge 2d +1$. More recently, in   \cite{AS} J. Azzam and R. Schul proved,
via the characterization of UR sets
by big pieces of bi-Lipschitz images, that every UR set is $\BP^2(\LG)$.

The purpose of this paper is to give a  far reaching generalization of these results and to prove that
``coronizations imply big pieces squared"
is a generic property. To give a first statement of our main result we consider, as we do  throughout the paper,  a fixed triple $(X,\dist,\mu)$  where $(X,\dist)$ is a metric space and $\mu$ is a Borel regular measure. To limit the number of parameters introduced in definitions and theorems we will for simplicity and consistently assume
that  $\diam(X) = \infty$: this assumption
is not essential (see Remark \ref{rlocal} below). We also fix a `dimension' $d \in (0,\infty)$. While  $(X,\dist,\mu)$ is fixed,
all constants appearing in our results will be independent of the particular metric measure space $(X,\dist,\mu)$ (while of course depending on the quantitative
parameters describing the space, e.g., the dimension $d$, the $d$-regularity constants, etc.). The following theorem, of which the precise statement  can be found at the beginning of Section \ref{proofCoBP.sect}, is  our main result.

\begin{theorem}\label{CoronaBP.thrm}
Let $E \subset X$ be a $d$-regular set with respect to the measure $\mu$ (see Definition \ref{regdef.def}). Suppose that $\mathcal{E}$ is a collection of closed subsets of $X$ each of which is $d$-regular with respect to the measure $\mu$ (with uniform bounds on the regularity constant). If $E$ admits a coronization with respect to $\mathcal{E}$ (see Definition \ref{coronadef.def}) then $E$ has big pieces squared of $\mathcal{E}$ (see Definition \ref{bigpiecesdef.def}).
\end{theorem}

We emphasize that although the formulation of Theorem \ref{CoronaBP.thrm}
does not require any particular quantitative restriction on the class $\mathcal{E}$, in
typical applications, the class $\mathcal{E}$ is subject to  some specified
quantitative control, and in this case the theorem says that the big pieces squared approximation is
obtained with respect to sets having the same (uniform) quantitative control.

In the classical Euclidean setting discussed above  $\mathcal{E}=\mathcal{E}^{Lip}_\eta$ and given that a $d$-regular
set $E\subset \rn, d<n$, has a corona decomposition with
respect to $\mathcal{E}^{Lip}_\eta$, we deduce from our Theorem \ref{CoronaBP.thrm}
that $E$ is approximable in the big pieces squared sense by Lipschitz graphs in the class
$\mathcal{E}^{Lip}_\eta$, for the specified $\eta>0$. In particular, based on characterization of UR sets in \cite{DS1},
we recover the result of \cite{AS} concerning big pieces squared approximability of uniformly
rectifiable sets by Lipschitz graphs.

An  alternate proof  of the result of  J. Azzam and R. Schul \cite{AS} in the case $d = n-1$,
based on corona-type constructions, was given by the first and third author in \cite{BH1}. While  Theorem \ref{CoronaBP.thrm}
applies in far more
general settings beyond the setting of UR sets in Euclidean spaces, a
consequence of Theorem \ref{CoronaBP.thrm}, and the characterization of UR sets by
coronizations with respect to Lipschitz graphs (see \cite{DS1}),  is that we here
provide a `corona analysis' type of proof of the result of J. Azzam and R. Schul \cite{AS} for $d < n$. However, it should  be noted that in their work
 \cite{AS} J. Azzam and R. Schul  also establish several other results beyond
the fact that  UR sets are $\BP^2(\LG)$.   Their work has been further
expanded upon by G. C. David and Schul \cite{GCDS}.

Another use of Theorem \ref{CoronaBP.thrm} is that it allows easy passage from a coronization to general `geometric lemmas' \cite{Jones,DS1,DS2}. It is a general fact that in the present setting (general) geometric lemmas are stable under the `big pieces functor' (in particular when applying it twice!). This big piece stability is just a matter of carefully checking that the proofs of David and Semmes \cite{DS2} and Rigot \cite{Rigot} adapt to our setting. Using Theorem \ref{CoronaBP.thrm} we can prove the following theorem and we refer to the bulk of the paper for  definitions of the geometric lemmas stated in the theorem.

\begin{theorem}\label{CoronaGLtrans.thrm}
Let $E \subset X$ be a $d$-regular set with respect to the measure $\mu$. Suppose that $\mathcal{E}$ is a collection of closed subsets of $X$ each of which is $d$-regular with respect to the measure $\mu$ with uniform bounds
on the regularity constant,
and that $\mathcal{A}$ is a collection of subsets of $X$ (not necessarily $d$-regular).
Suppose $E$ admits a coronization with respect to $\mathcal{E}$. Then the following implications hold:
\begin{itemize}
    \item If $p\in (0,\infty)$,  $q \in (0, \infty]$ satisfy
    \[\frac{1}{q} - \frac{1}{p} + \frac{1}{d} > 0,\]
     and if every $\widetilde{E} \in \mathcal{E}$ satisfies the $(p,q)$-geometric lemma with respect to $\mathcal{A}$, with uniform control on the Carleson measure constant, then $E$ satisfies the $(p,q)$-geometric lemma with respect to $\mathcal{A}$.
    \item If every $\widetilde{E} \in \mathcal{E}$ satisfies the weak geometric lemma with parameter $\epsilon$ with respect to $\mathcal{A}$, with uniform control on the Carleson set constant, then $E$ satisfies the weak geometric
    lemma with parameter $C\epsilon$ with respect to $\mathcal{A}$. Here $C$ depends only on dimension and the $d$-regularity constants.
    \item If every $\widetilde{E} \in \mathcal{E}$ satisfies the bilateral weak geometric lemma with parameter $\epsilon$ with respect to $\mathcal{A}$,  with uniform control on the Carleson set constant, then $E$ satisfies the bilateral weak geometric lemma with parameter $C\epsilon$ with respect to $\mathcal{A}$. Here $C$ depends only on dimension and the $d$-regularity constants.
\end{itemize}
\end{theorem}

We emphasize that all implications stated in Theorem \ref{CoronaGLtrans.thrm} are of a quantitative nature. The reader should also bear in mind that in the context of uniform rectifiability, the
collection $\mathcal{A}$ is the collection of all $d$-dimensional affine spaces. In the general setting of metric spaces there is no such analog, but the structure of the sets $\mathcal{A}$ is not important in the proofs given by David and Semmes \cite{DS2} and Rigot \cite{Rigot}.

While Theorem \ref{CoronaBP.thrm} and Theorem \ref{CoronaGLtrans.thrm} are, by their nature, very
general and of interest in many different contexts, one of our main
motivations is the application of these results in our ongoing project concerning
a parabolic version of parts of  \cite{DS1}, \cite{DS2}, with the goal of
establishing equivalent characterization of parabolic uniformly rectifiable
sets $E\subset\mathbb R^{n+1}$. In \cite{HLN}, \cite{HLN1} the third and fifth author, together with John Lewis,  introduced a  notion of parabolic uniformly rectifiable sets and proved, among other things, the existence of big pieces of regular parabolic Lipschitz graphs under the additional assumption that $E$ is Reifenberg flat in the parabolic sense. These studies were motivated by the study of parabolic or caloric measures in rough domains,
but up to now no systematic and correct study of parabolic uniformly rectifiable sets has
appeared in the literature. It is true that  in \cite{RN1, RN2,RN3},  the author took on the ambitious challenge to develop the theory of
parabolic uniformly rectifiable sets. Unfortunately though,  in \cite{RN1,RN2} the author either gives no proofs of his claims or supplies proofs which have gaps, a few of which we pinpoint in \cite{BHHLN}. In \cite{RN3} the author 
states that the parabolic corona decomposition implies parabolic UR, 
with a proof going through the corresponding `alpha' 
numbers as in \cite{To}. On the other hand, this result is also
a corollary of our Theorem \ref{CoronaGLtrans.thrm} (see 
Theorem \ref{t4.6}), and as our proof is based on an entirely different method,  
we have not
checked in detail the validity of the method claimed in \cite{RN3}. 

In forthcoming papers, 
including  \cite{BHHLN}, along with the present paper,
we conduct a thorough study of parabolic
uniformly rectifiable sets,
and in the context of this paper we note that in \cite{BHHLN} we prove,
among other
things, that parabolic uniformly rectifiable sets (see Definition \ref{p-UR}), satisfy a corona decomposition
with respect to {\em regular} Lip(1,1/2) graphs (see Definition \ref{defgpg}).   Such graphs are the natural
parabolic analogues of Lipschitz graphs,  from the point of view of both
singular integral theory, and PDE/potential theory (see \cite{H1,H,HL,LM,KW}).
In the present paper, we obtain a converse to this  result from  \cite{BHHLN},
as we prove that corona decomposition
with respect to {\em regular} Lip(1,1/2) graphs implies parabolic uniformly rectifiability.
This converse
is a rather  straightforward consequence of the general results established in this paper;
we refer to Section \ref{parabolic.sect}
for details, see
in particular Theorem \ref{t4.6} (i).
In combination,
the present paper and
 \cite{BHHLN} show that, just as in the elliptic
 setting \cite{DS1},  we can characterize
 parabolic uniform rectifiability in terms of the existence
of a corona decomposition with
respect to an appropriate family of graphs ({\em regular} Lip(1,1/2) graphs). 
We further obtain that all sufficiently ``nice" parabolic singular integral operators are $L^2$ bounded on a parabolic uniformly rectifiable set; see Theorem \ref{t4.16} and Corollary \ref{c4.17} below.

The rest of the paper is organized as follows. Section \ref{Prelim} is of preliminary nature.  Theorem \ref{CoronaBP.thrm} is proved in Section \ref{proofCoBP.sect} and the proof is based on an induction argument. Theorem \ref{CoronaGLtrans.thrm} is a consequence of Theorem 1.1, and
Propositions \ref{geolemprop1.prop}, \ref{geolemprop2.prop} and \ref{geolemprop3.prop} which
establish stability of
various `geometric lemmas' in this general setting, and are stated in the bulk of the paper. In fact, minor modifications
aside, the proofs of the three propositions follow almost exactly the corresponding
proofs in \cite{DS2,Rigot}. In this sense we
claim little originality in this part and we therefore
postpone  the proofs (or perhaps rather the confirmations of the validity) of  Propositions \ref{geolemprop1.prop}, \ref{geolemprop2.prop} and \ref{geolemprop3.prop} to an appendix at the end of the paper, Appendix \ref{geolem.app}. However, these proposition are used in Section \ref{parabolic.sect}
where we detail and prove our applications to parabolic uniform rectifiability and we note that Propositions \ref{geolemprop1.prop}, \ref{geolemprop2.prop} and \ref{geolemprop3.prop} have previously not appeared in the literature in the context of  parabolic uniform rectifiability.

\section{Preliminaries}\label{Prelim}
Recall $(X,\dist,\mu)$ and $d$ introduced in the introduction.  In the sequel,
$B(x,r)$, for $x \in X$ and $r > 0$,
will always denote the usual metric ball defined with respect to $\dist$ and centered at $x$ with radius $r$.

As is customary, we use the letters $c,C$ to denote harmless positive constants, not necessarily
the same at each occurrence, which depend only on dimension ($d$) and the
constants appearing in the hypotheses of theorems/lemmas (which we refer to as the
``allowable parameters'').  In some cases, we shall simply use the letter $C$ to denote one of these fixed
allowable parameters (see, e.g., Definition \ref{regdef.def} below).
We shall also
sometimes write $a\lesssim b$ and $a \approx b$ which mean, respectively,
that $a \leq C b$ and $0< c \leq a/b\leq C$, where the constants $c$ and $C$ are, unless otherwise stated, as above. When a constant is given a numerical subscript (e.g. $C_0$) its value will be fixed.

\begin{definition}[$d$-regularity]\label{regdef.def}
Let $E \subset X$. We say $E$ is $d$-regular (with respect to $\mu$) up to
scale $R_0 \in (0,\infty]$ and with constant $C >1$, written $E \in \Reg(C,R_0)$ if $E$ is closed and
\[C^{-1}r^d \le \mu(B(x,r) \cap E) \le Cr^d, \quad \forall x\in E, r \in (0,R_0). \]
We call the upper bound here the {\bf upper regularity condition} and the lower bound here the {\bf lower regularity condition}.
In the case $R_0 = \infty$ we simply write $E \in \Reg(C)$.
\end{definition}

\subsection{Trading for scales}

The following lemma allows us to localize any $d$-regular set.

\begin{lemma}\label{localregset.lem}
Let $E \in\Reg(C)$. Then for every $x \in E$ and $r> 0$ there exists $E_{x,r} \subset E$ such that $E_{x,r} \in \Reg(2^{6d}10^dC,10r)$ and
\[B(x,r) \cap E \subset E_{x,r} \subset B(x,3r) \cap E.\]
In particular, $\diam(E_{x,r}) \ge C^{-2/d}(r/2)$.
\end{lemma}
\begin{proof}
The statement about the diameter of $E_{x,r}$ immediately follows from the regularity of $E$ and that $B(x,r) \cap E \subset E_{x,r} \subset B(x,3r) \cap E$. Indeed,
\[\mu(B(x,C^{-2d}r/2) \cap E) \le 2^{-d} \mu(B(x,r) \cap E),\]
since the the right hand side is non-zero this implies there exists $y \in E \cap B(x,r) \setminus B(x,C^{-2d}r/2)$, which immediately gives the diameter estimate.

Now we produce the set $E_{x,r}$. Let $A_0 = B(x,r) \cap E$ and for $k = 1,2,\dots$ we defined $A_k$ inductively by
\[A_k = \bigcup_{z \in A_{k-1}} B(z, 2^{-k} r) \cap E.\]
Set $A = \bigcup_{k \ge 0} A_k$. Obviously $B(x,r) \cap E \subseteq A$.

Let $z \in A$ be fixed. Then $z \in A_{k_0}$ for some $k_0$ and by definition there exists $z_0, z_1, \dots, z_{k_0 -1}$ such that $z_k \in A_k$, $\dist(z_0, x) < r$, $\dist(z_{k_0 - 1}, z) < 2^{-k_0} r$ and
\[\dist(z_k, z_{k +1}) < 2^{-k-1}r.\]
It immediately follows from the triangle inequality that
\[\dist(z,x) < \sum_{k = 0}^\infty 2^{-k}r = 2r,\]
which gives that $A \subseteq B(x,2r) \cap E$.
Let $s \in (0, 2^{-k_0 + 5}r]$. Using that $B(z,2^{-6}s) \cap E \subset B(z, 2^{-k_0-1}r) \cap E \subseteq A$
we have
\begin{align*}
    C^{-1}2^{-6d}s^d \le \mu\big(E \cap B(z,2^{-6}s)\big) &= \mu\big(B(z,2^{-6}s) \cap A\big)\\
    & \le \mu(E \cap B(z,s)) \le Cs^d,
\end{align*}
where we used the $d$-regularity of $E$ in the first and last inequalities.
Now suppose that $s \in [2^{-j - 1}r, 2^{-j}r)$ for some $j \in \{0, 1, \dots k_0 -6\}$.
Then with $\{z_k\}_{k = 0}^{k_0 - 1}$ as above we have that $\dist(z_{k},z) \le 2^{-k}r$, so that
$B(z_{j+5}, 2^{-j-6}r) \cap E \subset B(z,s) \cap A$. Thus,
\begin{align*}
    C^{-1}2^{-6d}s \le C^{-1}2^{(-j - 6)d}r
    &\le \mu( B(z_{j+5}, 2^{-j-6}r) \cap E)\\
 &\le \mu(B(z,s) \cap A) \le Cs^d,
\end{align*}
where we used the regularity of $E$ and that $A \subset E$ in the last line. If $s \in [r,10r)$,
then $s'= s/10 \in (0,r)$, so appealing to the analysis above we obtain
\[C^{-1}10^{-d}s \le \mu(B(z,s/10) \cap A) \le \mu(B(z,s) \cap E) \le Cs^d.\]

This shows that
\[ C^{-1} 10^{-d}2^{-6d}s \le \mu(B(z,s) \cap A) \le Cs^d, \quad \forall z \in A, s \in (0,10r).\]

Notice that no point of $A$ is isolated.  We take $E_{x,r}$ to be the closure of $A$, then since $E$ is closed $E_{x,r} \subseteq E$. If $w \in E_{x,r}$, $\epsilon \in (0,1/2)$ and $s \in (0,10r)$ then there exists $z \in A$ such that $\dist(z,w) \le \epsilon s$ and hence
\begin{multline*}
    C^{-1}10^{-d}2^{-6d}(1 - \epsilon)^ds^d
    \le \mu(A \cap B(z,(1-\epsilon)s))
    \\ \le \mu(E_{x,r} \cap B(w,s))
    \le \mu(E \cap B(w,s)) \le C s^d.
\end{multline*}
which gives that $E_{x,r} \in \Reg(2^{6d}10^{d}C,10r)$. The fact that $B(x,r) \cap E \subset E_{x,r} \subset B(x,3r)$ follows from the analysis above as well.
\end{proof}

\begin{remark}[``Trading For Scales"]\label{trading.rmk}
 In the proof of Lemma \ref{localregset.lem} we used a technique which one might call ``trading for scales", where we sacrifice some portion of a structural constant in order to gain in `scale'. This can also be done with the constants in the big pieces definition (see Definition \ref{bigpiecesdef.def}) and is demonstrated in Lemma \ref{cubesenoughbp.lem}. This idea will be used frequently in the proof of Theorem \ref{CoronaBP.thrm} and, due to the focus on other technical matters, at that time we will use this technique without mentioning it at each occurrence.
\end{remark}

As an example of trading for scales we produce the following lemma, which is applicable to the set constructed in Lemma \ref{localregset.lem}.

\begin{lemma}\label{regscalechange.lem}
Suppose that $E \in \Reg(C,R)$ with $R > \diam E$. Then
$$E \in \Reg(C(R'/R)^d,R')\,,$$ for all $R' > R$.
\end{lemma}
\begin{proof}
If $r \in (0,R)$ then $C^{-1}r^d \le \mu(E \cap B(x,r)) \le Cr^d$ for all $x \in E$. If $r \in [R,R')$ then $r > \diam E$ and hence for $x \in E$
\begin{align*}\mu(E \cap B(x,r)) = \mu(E \cap B(x,R)) &\ge C^{-1}R^d\\
& = C^{-1}(R'/R)^d (R')^d \ge C^{-1}(R'/R)^d r^d.
\end{align*}
Additionally, for $r \in [R,R')$ and $x \in E$ it holds $$\mu(E \cap B(x,r)) = \mu(E \cap B(x,R)) \le CR^{d} \le Cr^d.$$ These estimates give the lemma.
\end{proof}

\subsection{Dyadic notation}

\begin{lemma}\label{lemmaCh}\textup{({\bf Existence and properties of the ``dyadic grid''})
\cite{DS1,DS2}, \cite{Ch}, \cite{hyt-k}.}
Suppose that $E \in \Reg(C)$ Then there exist
constants $ a_0>0,\, \gamma>0$ and $C_1<\infty$, depending only on $d$ and $C$, such that for each $k \in \mathbb{Z},$
there is a collection of pairwise disjoint Borel sets (``cubes'')
$$
\mathbb{D}_k:=\{Q_{j}^k\subset E: j\in \mathfrak{I}_k\},$$ where
$\mathfrak{I}_k \subseteq \mathbb{N}$ denotes some index set depending on $k$, satisfying

\begin{list}{$(\theenumi)$}{\usecounter{enumi}\leftmargin=.8cm
\labelwidth=.8cm\itemsep=0.2cm\topsep=.1cm
\renewcommand{\theenumi}{\roman{enumi}}}

\item $E=\cup_{j}Q_{j}^k\,\,$ for each
$k\in{\mathbb Z}$.

\item If $m\geq k$ then either $Q_{i}^{m}\subset Q_{j}^{k}$ or
$Q_{i}^{m}\cap Q_{j}^{k}=\emptyset$.

\item For each $(j,k)$ and each $m<k$, there is a unique
$i$ such that $Q_{j}^k\subset Q_{i}^m$.

\item $\diam\big(Q_{j}^k\big)\leq C_1 2^{-k}$.

\item Each $Q_{j}^k$ contains some ``surface ball'' $\Delta \big(x^k_{j},a_02^{-k}\big):=
B\big(x^k_{j},a_02^{-k}\big)\cap E$.
\end{list}
\end{lemma}
A few remarks are in order concerning this lemma.
\begin{list}{$\bullet$}{\leftmargin=0.4cm  \itemsep=0.2cm}

\item In the setting of a general space of homogeneous type, this lemma was proved by Christ
\cite{Ch}, with the
dyadic parameter $1/2$ replaced by some constant $\delta \in (0,1)$.
In fact, one may always take $\delta = 1/2$ (see  \cite[Proof of Proposition 2.12]{HMMM}).
In the presence of the Ahlfors-David
property, and in Euclidean space the result already appears in \cite{DS1,DS2}.

\item  We shall denote by  $\mathbb{D}=\mathbb{D}(E)$ the collection of all relevant
$Q^k_j$, i.e., \[\mathbb{D} := \cup_{k} \mathbb{D}_k.\]

\item Properties $(iv)$ and $(v)$ imply that for each cube $Q\in\mathbb{D}_k$,
there is a point $x_Q\in E$, a metric ball $B(x_Q,r)$ and a surface ball
$\Delta(x_Q,r):= B(x_Q,r)\cap E$ such that
$r\approx 2^{-k} \approx {\rm diam}(Q)$
and \begin{equation}\label{cube-ball}
\Delta(x_Q,r)\subset Q \subset \Delta(x_Q,Cr),\end{equation}
for some uniform constant $C$.
We shall denote this ball and surface ball by
\begin{equation}\label{cube-ball2}
B_Q:= B(x_Q,r) \,,\qquad\Delta_Q:= \Delta(x_Q,r),\end{equation}
and we shall refer to the point $x_Q$ as the ``center'' of $Q$.

\item For a dyadic cube $Q\in \mathbb{D}_k$, we shall
set $\ell(Q) = 2^{-k}$, and we shall refer to this quantity as the ``length''
of $Q$.  Evidently, $\ell(Q)\approx \diam(Q).$

\item For a dyadic cube $Q \in \mathbb{D}$ and $K > 1$ we define
\[KQ:= \{x \in E: \dist(x,Q) \le (K - 1) \diam(Q)\}.\]
\end{list}

\begin{definition}[Localized Dyadic Grids and Sawtooths]
Let $C > 1$, $E \in \Reg(C)$ and $\mathbb{D} = \mathbb{D}(E)$ as above. For $Q \in \mathbb{D}$ we set $\mathbb{D}_{Q} = \{Q' \in \mathbb{D}: Q'\subseteq Q\}$. If $\mathcal{F} = \{Q_j\}$ is a countable collection of pairwise disjoint cubes in $\mathbb{D}$ then we set $\mathbb{D}_{\mathcal{F}} = \mathcal{D} \setminus \cup_j \mathbb{D}_{Q_j}$. If $Q \in \mathbb{D}$ and  $\mathcal{F} = \{Q_j\}$ is a countable collection of pairwise disjoint cubes in $\mathbb{D}$ then we set
\[\mathbb{D}_{\mathcal{F}, Q} =\mathbb{D}_{Q} \cap \mathbb{D}_{\mathcal{F}}. \]
\end{definition}

\subsection{Carleson measures and decompositions}

\begin{definition}[Discrete Measures and Discrete Carleson Norms]
Suppose $C > 1$, $E \in\Reg(C)$ and $\mathbb{D} = \mathbb{D}(E)$ be as above. Let $\{\alpha_Q\}_{Q \in \mathbb{D}}$, where $\alpha_Q \in [0, \infty)$.
We let $\mut$ be the discrete measure associated to $\{\alpha_Q\}_{Q \in \mathbb{D}}$ be defined by
\[\mut(\mathbb{D}') = \sum_{Q \in \mathbb{D}'} \alpha_Q,\]
for any collection of cubes $\mathbb{D}' \subseteq \mathbb{D}$. If $\mathcal{F} = \{Q_j\}$ is a countable collection of pairwise disjoint cubes in $\mathbb{D}$ we define $\mut_{\mathcal{F}}$ by
\[\mut_{\mathcal{F}}(\mathbb{D}') = \mut(\mathbb{D}' \cap \mathbb{D}_{\mathcal{F}}).\]
If $\mathcal{F} = \{Q_j\}$ is a countable collection of pairwise disjoint cubes in $\mathbb{D}$ we define the global Carleson norm of $\mut_{\mathcal{F}}$ as
\[\|\mut_\mathcal{F}\|_{\mathcal{C}} = \sup_{Q \in \mathbb{D}} \frac{\mut_\mathcal{F}(\mathbb{D}_Q)}{\mu(Q)}\]
and for $Q_0 \in \mathbb{D}$ the localized Carleson norm of $\mut_\mathcal{F}$ (with respect to $Q_0$) as
\[\|\mut_\mathcal{F}\|_{\mathcal{C}(Q_0)} = \sup_{Q \in \mathbb{D}_{Q_0}} \frac{\mut_\mathcal{F}(\mathbb{D}_Q)}{\mu(Q)}.\]
Here if $\mathcal{F}= \emptyset$ we write $\mut$ in place of $\mut_{\mathcal{F}}$ in the notation above.
\end{definition}

An important ingredient in the proof of Theorem \ref{CoronaBP.thrm} is the
following decomposition of a discrete Carleson region.

\begin{lemma}\label{extraplem.lem}{\cite[Lemma 7.2]{HM-I}}
Suppose that $C' > 1$, $E \in \Reg(C')$ and let $\mathbb{D}(E)$ be as above. Suppose that $\mut$ is a discrete measure associated to $\{\alpha_Q\}_{Q \in \mathbb{D}}$. There exists $C$ depending on $d$ and $C'$ such that the following holds. Given $a\geq 0$, $b>0$, and $Q \in \mathbb{D}$ such that
$\mut(\dd_{Q})\leq (a+b)\,\mu(Q)$,
there is a family $\F=\{Q_j\}\subset\dd_{Q}$
of pairwise disjoint cubes such that
\begin{equation} \label{Corona-sawtooth}
\|\mut_\F\|_{\C(Q)}
\leq C b,
\end{equation}
\begin{equation}
\label{Corona-bad-cubes}
\mu(B)
\leq \frac{a+b}{a+2b}\, \mu(Q)\,,
\end{equation}
where $B$ is the union of those $Q_j\in\F$ such that
$\mut\big(\dd_{Q_j}\setminus \{Q_j\}\big)>a\,\mu(Q_j)$.
\end{lemma}

\subsection{Corona decompositions and big pieces}

Before we introduce the notion of corona decomposition we need the following definition.

\begin{definition}[\cite{DS2}]\label{d3.11}
Suppose $E$ is $d$-regular with dyadic cubes $\dd(E)$. Let $\sbf\subset \dd(E)$. We say that $\sbf$ is
``coherent" if the following conditions hold:
\begin{itemize}\itemsep=0.1cm

\item[$(a)$] $\sbf$ contains a unique maximal element $Q(\sbf)$ which contains all other elements of $\sbf$ as subsets.

\item[$(b)$] If $Q$  belongs to $\sbf$, and if $Q\subset \widetilde{Q}\subset Q(\sbf)$, then $\widetilde{Q}\in {\bf S}$.

\item[$(c)$] Given a cube $Q\in \sbf$, either all of its children belong to $\sbf$, or none of them do.

\end{itemize}
We say that $\sbf$ is ``semi-coherent'' if only conditions $(a)$ and $(b)$ hold.
\end{definition}

\begin{definition}[Coronizations]\label{coronadef.def} Let $C_*,C_{**} > 1$. Suppose that $E\in \Reg(C_*)$ and let $\mathbb{D}(E)$ be as above. Let $\mathcal{E} \subseteq \Reg(C_{**})$.
Suppose
$0 < \eta$
and $K>  1$. We say that $E$ admits an $(\eta, K)$-coronization with respect to $\mathcal{E}$ if the following holds. There is a disjoint decomposition
$\dd(E) = \G\cup\B$, satisfying the following properties.
\begin{enumerate}
\item  The ``Good"collection $\G$ is further subdivided into
disjoint stopping time regimes $\{\sbf\}_{\sbf \in \mathcal{S}}$, such that each such regime $\sbf \in \mathcal{S}$ is coherent (cf. Definition \ref{d3.11}).

\item The ``Bad" cubes, as well as the maximal cubes $Q(\sbf)$ satisfy a Carleson
packing condition: There exists a constant $C_{\eta,K} \ge 0$ such that
$$\sum_{Q'\subseteq Q, \,Q'\in\B} \mu(Q')
\,\,+\,\sum_{\sbf: Q(\sbf)\subseteq Q}\mu\big(Q(\sbf)\big)\,\leq\, C_{\eta,K}\, \mu(Q)\,,
\quad \forall Q\in \dd(E)\,.$$
\item For each $\sbf$, there exists $\Gamma_{\sbf} \in \mathcal{E}$ for every $Q\in \sbf$,
\begin{equation}\label{eq2.2a}
\sup_{x\in KQ} \dist(x,\Gamma_{\sbf} ) < \eta\,\ell(Q).
\end{equation}
\end{enumerate}
In the sequel, we write $\mathcal{M} = \{Q(\sbf)\}_{\sbf \in \mathcal{S}}$ to denote the set of maximal cubes.
\end{definition}

\begin{definition}[Big Pieces]\label{bigpiecesdef.def}
Suppose that $C, C', C''> 1$, $\theta, \theta' > 0$. Suppose that $\mathcal{E} \subseteq \Reg(C)$.
We say that $E \in \Reg(C')$ has big pieces of $\mathcal{E}$ with constant $\theta$,
written $E \in \BP(\mathcal{E})(\theta, C')$,
if for for every $x \in E$ and $r > 0$ there exists $\Gamma \in \mathcal{E}$ such that
\[\mu\big(\Gamma \cap E \cap B(x,r)\big) \ge \theta r^d.\] We say that $E \in \Reg(C'')$ is in  $\BP(\BP(\mathcal{E})(\theta, C'))(\theta', C'')$, if for every $x \in E$ and $r > 0$ there exists $\Gamma \in \Reg(C') \cap \BP(\mathcal{E})(\theta, C')$ such that
\[ \mu\big(\Gamma \cap E \cap B(x,r)\big) \ge \theta' r^d.\]
\end{definition}

\begin{remark}\label{rlocal}
It is implicit in the preceding definition that $\diam(E)=\infty$.  
We shall work with unbounded sets $E$ (except when utilizing the localization Lemma \ref{localregset.lem}),
for the sake of convenience, but this is a minor matter.
 In Euclidean space, if the property in question holds in particular 
for $d$-planes, then there is a standard procedure to treat the case of bounded sets:  if
$E$ is a bounded set 
satisfying a corona decomposition with respect to some class of sets $\mathcal{E}$, 
and if $d$-planes also enjoy the corona property with respect to the same
class $\mathcal{E}$, then we may consider the set $E_* = E \cup P$, where $P$ is a $d$-plane whose distance to $E$ 
is comparable 
to the diameter of $E$.  Then Theorem \ref{CoronaBP.thrm} 
says that $E_* \in BP^2(\mathcal{E})$, and hence $E$ inherits the $BP^2$ property.
 
More generally, in a metric space setting, where the notion of a linear subspace does not exist, 
one can modify our proofs {\it mutatis mutandis} to the bounded setting.  We leave the details to the interested reader.
\end{remark}

\begin{remark}\label{convention}
Below we will always establish results where the values $$C,C',C'',\theta,\theta'$$ are all controlled by the allowable parameters. For this reason and to ease notation we will often simply write $\BP(\mathcal{E})$ in place of $\BP(\mathcal{E})(\theta, C')$ and $\BP^2(\mathcal{E})$ in place of $\BP(\BP(\mathcal{E})(\theta, C'))(\theta', C'')$. We are quite sure that the reader will appreciate this.
\end{remark}

Trading for scales (see Remark \ref{trading.rmk})  and using the dyadic cube construction allows us to
check the big pieces condition only on dyadic cubes.

\begin{lemma}[Big pieces on cubes is big pieces]\label{cubesenoughbp.lem}
Let $\theta > 0$. Let $E \in \Reg(M)$ and $\mathcal{E} \subset \Reg(L)$ for some $M, L > 0$. There is a constant $c = c(d, M)$, such that
 if for every cube $Q \in \mathbb{D}(E)$, there exists $\Gamma \in \mathcal{E}$, with
\[\mu(Q \cap \Gamma) \ge \theta\mu(Q)\,,\]
then $E \in \BP(\mathcal{E})(c\theta,M)$. Conversely, there is a constant $c' = c'(M,d)$,
such that if $E \in \BP(\mathcal{E})(\theta,M)$,
then for every cube $Q \in \mathbb{D}(E)$, there
exists $\Gamma \in \mathcal{E}$ such that
\[\mu(Q \cap \Gamma) \ge c'\theta\mu(Q).\]
\end{lemma}
\begin{proof}
Suppose that $Q \in \mathbb{D}(E)$ there exists $\Gamma$ such that $\mu(Q \cap \Gamma) \ge \theta\mu(Q)$. Let $x \in E$ and $r > 0$. Recall the dyadic cubes have the property that $\diam(Q) \approx \ell(Q) \approx \mu(Q)^{1/d}$ and $\cup_{Q \in \mathbb{D}_k} = Q$ with $Q \in D_k$ meaning $\ell(Q) = 2^{-k}$. Then choosing $k$ such that $ k \ll \log_2 r \lesssim  k$ there is a cube $Q \in \mathbb{D}_k$ with $x \in Q$ and $\diam(Q) \approx r$. Thus, $Q \subset B(x,r) \cap E$ and there exists $\Gamma \in \mathcal{E}$ such that \[\mu(B(x,r) \cap \Gamma) \ge \mu(Q \cap \Gamma) \ge \theta \mu(Q) \approx \theta \ell(Q)^d \approx \theta r^d.\]
As the implicit constants above only depend on $d$ and $M$ it follows that $E \in \BP(\mathcal{E})(c\theta, M)$ for some $c = c(d,M) > 0$.

Now suppose that $E \in \BP(\mathcal{E})(\theta, M)$. By the properties of dyadic cubes, there exists $a_0 > 0$ depending only on $d$ and $M$ such that for any cube $Q \in \mathbb{D}(E)$, $B(x_Q, a_0 \ell(Q))\cap E \subseteq Q$. Let $Q \in \mathbb{D}$ then by hypothesis there exists $\Gamma \in \mathbb{E}$ such that
\[\mu(Q \cap \Gamma) \ge \mu(B(x_Q, a_0 \ell(Q))\cap E) \ge \theta (a_0 \ell(Q))^d) \approx \theta \ell(Q)^d \approx \theta \mu(Q).\]
This proves the lemma.
\end{proof}

\subsection{$\beta$-numbers and geometric lemmas}

We now state the definitions of some geometric lemmas. Here the words `geometric lemmas' mean unilateral or bilateral closeness of a regular set to a family of sets (which are not necessarily regular), quantified in terms of a Carleson measure or Carleson set condition.

\begin{definition}[$\beta$-numbers for general sets]
Let $\mathcal{A}$ be an arbitrary collection of (non-empty) sets. Fix $E \in \Reg(C)$. For $q \in (0, \infty)$, $Q\in \mathbb{D}(E)$ we define
\[\beta_{q,\mathcal{A}}(Q):= \inf_{A \in \mathcal{A}} \left\{\mu(Q)^{-1}\int_{2Q} \left[(\diam Q)^{-1} \dist(y, A)\right]^{q} \, d\mu(y) \right\}^{1/q}.  \]
and when $q = \infty$ we define for $Q$ a dyadic cube
\[\beta_{\mathcal{A}}(Q) := \beta_{\infty, \mathcal{A}}(Q)
= \inf_{A \in \mathcal{A}}\left\{ \diam(Q)^{-1} \sup_{y \in 2Q} \dist(y,A) \right\}.\]
\end{definition}

\begin{definition}[$(p,q)$-general geometric lemmas]
 For fixed $p \in (0, \infty)$ and $q \in (0,\infty]$ we say that $E$ satisfies the $(p,q)$-general geometric lemma with respect to $\mathcal{A}$ written $E \in \GLem(\mathcal{A},p, q)$ if there exists $M >0 $ such that
\[\sum_{Q \subseteq R} [\beta_{q,\mathcal{A}}(Q)]^p \mu(Q) \le M\mu(R), \quad  R \in \mathbb{D}(E).\]
At times, we shall want to stress the {\bf Carleson measure constant} $M$ and we then
write $E \in \GLem (\mathcal{A},p,q,M)$.
\end{definition}

\begin{definition}[The weak geometric lemma]
Given $\epsilon > 0$ we say that $E$ satisfies the weak geometric lemma with parameter $\epsilon$ with respect to $\mathcal{A}$, written $E \in \WGLem (\mathcal{A},\epsilon)$ if there exists $M_\epsilon > 0$  such that
\[\sum_{\substack{Q \subseteq R\\ \beta_{\mathcal{A}}(Q) > \epsilon}} \mu(Q) 
\, \le\, M_\epsilon\,\mu(R), \quad  R \in \mathbb{D}(E).\]
At times, we shall
want to stress the {\bf Carleson set constant}, and we then
write $E \in \WGLem (\mathcal{A},\epsilon,M_\epsilon)$.
\end{definition}

\begin{definition}[Bilateral versions and the bilateral weak geometric lemma]\label{BWGL.def}
We also define a bilateral version of $\beta_\infty$ for
any dyadic cube $Q$ as
\begin{equation*}\label{bibetadef.def} b\beta_{\mathcal{A}}(Q) :=
\diam(Q)^{-1}\inf_{A \in \mathcal{A}}\left\{  \sup_{y \in 2Q} \dist(y,A)
\, + \sup_{ z \in A \cap B(x_Q,2\diam(Q))}  \dist(z,E) \right\}\,,
\end{equation*}
where $x_Q$ is the `center' of $Q$, as in  \eqref{cube-ball}.
We say $E$ satisfies the bilateral weak geometric lemma with parameter $\epsilon$ with respect to $\mathcal{A}$, written $E \in \BWGLem (\mathcal{A},\epsilon)$ if there exists $M_\epsilon > 0$ such that
\[\sum_{\substack{Q \subseteq R\\ b\beta_{\mathcal{A}}(Q) > \epsilon}} \mu(Q) \,\le \, 
\, M_\epsilon \,\mu(R), \quad  R \in \mathbb{D}(E).\]
We shall write $E \in \BWGLem (A,\epsilon,M_\epsilon)$
when we want to stress the {\bf Carleson set constant}.
\end{definition}

\begin{remark} The `dilation parameter' 2 in the definitions of $\beta$ and $b\beta$
could be replaced by any
$\kappa \geq 2$, i.e., with $\kappa Q,\, \kappa \diam(Q)$ in place of $2Q, \, 2\diam(Q)$.
\end{remark}

\begin{remark}\label{empty.rmk}
Concerning the definition of $b\beta(Q)$,  if the set
 \begin{equation}\label{eq2.21}
     A \cap B(x_Q,2\diam(Q)) = \emptyset\,,
\end{equation}
with $x_Q$ as in  \eqref{cube-ball},
then we set $\sup_{ z \in A \cap B(x_Q,2\diam(Q)) } \dist(z,E) = 0$. This is not a problem in applications as $\epsilon$ in the definition of $\BWGLem (A,\epsilon,M_\epsilon)$ is typically (very) small,
and when \eqref{eq2.21} holds, the first term in $b\beta$ is greater than or
equal to $1$. For instance, the membership $E \in \BWGLem (\mathcal{A},\epsilon,M_\epsilon)$ and $E \in \WGLem (\mathcal{A},\epsilon,M_\epsilon)$ hold vacuously for any regular set $E$ whenever $\epsilon > 2$ and $\mathcal{A}$ is any collection of sets such that for every $x \in X$ there exists $A \in \mathcal{A}$ such that $x \in A$ (e.g. $X = \rn$ and $\mathcal{A}$ is the collection of all $d$-dimensional affine spaces).
\end{remark}

\begin{remark}\label{parameters.rmk}
In the literature, the weak and bilateral weak geometric lemmas are often stated
in a `parameterless' manner. In particular, we say the weak geometric lemma holds for $E$ if there is a function $\gamma: (0, 1] \to \re$ such that $E \in \WGLem (\mathcal{A}, \epsilon, \gamma(\epsilon))$,
for every $\epsilon \in (0,1]$.
\end{remark}

\begin{remark}
 We would like to point out that the `choice' of dyadic grid is not important in the definitions of the geometric lemmas, provided one is willing to lose something in the parameters ($\epsilon$ and the Carleson constants). For instance, if $E \in \GLem (\mathcal{A}, p,q,M)$ with respect to some grid $\mathbb{D}$ then $E \in \GLem (\mathcal{A}, p,q,M')$ for any other grid $\widetilde{\mathbb{D}}$, where $M'$ depends only on $M$, $d$, and the regularity of $E$.
\end{remark}

\subsection{Stability of geometric lemmas under the `big piece functor'}
We here   state three propositions concerning the stability of geometric lemmas defined in the previous subsction under the `big piece functor'. The proofs of these propositions can be found in Appendix \ref{geolem.app}.  Concerning the general geometric lemma, weak geometric lemma and bilateral weak geometric lemma the following hold.
\begin{proposition}\label{geolemprop1.prop}
Let $\mathcal{A}$ be a collection of subsets of $X$. Let $C_*> 1$, $\theta, M > 0$ and $p \in (0, \infty)$, $q \in (0,\infty]$ satisfy
\[\frac{1}{q} - \frac{1}{p} + \frac{1}{d} > 0.\]
Suppose that $E \in \Reg(C_*)$
and that $$\mathcal{E} \subset \Reg(C_*) \cap \GLem (\mathcal{A},p, q,M).$$
If $E \in \BP(\mathcal{E})(\theta, C)$ then $E \in \GLem (\mathcal{A},p, q,M')$, where $M'$ depends on $C_*, \theta, M, p, q$ and dimension.
\end{proposition}

\begin{proposition}\label{geolemprop2.prop}
Let $\mathcal{A}$ be a collection of subsets of $X$. Let $C_* > 1$, $\epsilon, \theta, M> 0$. Suppose that
$E \in \Reg(C_*)$ and $$\mathcal{E} \subset \Reg(C_*) \cap \WGLem (\mathcal{A},\epsilon,M).$$ If
$E \in \BP(\mathcal{E})(\theta, C_*)$ then $E \in \WGLem (\mathcal{A},C\epsilon,M')$, where $C$ depends only on dimension and $C_*$ and $M'$ depends on $C_*$, $\epsilon, \theta, M> 0$ and dimension.
\end{proposition}

\begin{proposition}\label{geolemprop3.prop}
Let $\mathcal{A}$ be a collection of subsets of $X$. Let $C_* > 1$, $\epsilon, \theta, M> 0$.
Suppose that $E \in \Reg(C_*)$ and $$\mathcal{E} \subset \Reg(C_*) \cap \BWGLem (\mathcal{A},\epsilon,M).$$ If
$E \in \BP(\mathcal{E})(\theta, C_*)$  then $E \in \BWGLem (\mathcal{A},C\epsilon,M')$, where $C$ depends only on dimension and $C_*$ and $M'$ depends on $C_*$, $\epsilon, \theta, M> 0$ and dimension.
\end{proposition}

\begin{remark}\label{r2.31}
We remark that in the statements of the preceding propositions, there is no loss of generality
to assume that the $d$-regularity constants for $E$ and $\mathcal{E}$ are the same,
as in general we may simply take the larger of the two.   This same remark applies in the sequel.
\end{remark}

\section{Proof of Theorem \ref{CoronaBP.thrm}}\label{proofCoBP.sect}
In this section we prove Theorem \ref{CoronaBP.thrm}. Using the notation introduced in the previous section the following is  the precise formulation of  Theorem \ref{CoronaBP.thrm} and this is the statement that  we will prove.
We observe that Remark \ref{r2.31} applies here.

\smallskip

\noindent {\it Let $C_*> 1$. Suppose that $E\in \Reg(C_*)$ and let $\mathbb{D}=\mathbb{D}(E)$ be the set of dyadic cubes as Lemma \ref{lemmaCh}.
Assume that $\mathcal{E} \subseteq \Reg(C_{*})$ and that $E$ admits an $(\eta, K)$-coronization (see Definition \ref{coronadef.def})
 with respect to $\mathcal{E}$ for some $\eta > 0$ and $K > 1$. Then $E \in \BP^2(\mathcal{E})$ with constants depending on $C_*, d, \eta, K$ and $C_{\eta,K}$ (the constant in Definition \ref{coronadef.def}). }

\smallskip

From this point forward we assume that $E\in \Reg(C_*)$, $\mathcal{E} \subseteq \Reg(C_{*})$ and that $E$ admits an
$(\eta, K)$-coronization
with respect to $\mathcal{E}$ for some $\eta > 0$ and $K > 1$. We define
\begin{equation}\label{eq4.0}
\alpha_Q:= 
\begin{cases} \mu(Q)\,,&{\rm if}\,\, Q\in \M\cup\B, \\
0\,,& {\rm otherwise}.\end{cases}
\end{equation}
and we let $\mut$ be the discrete measure with respect to $\{\alpha_Q\}_{Q \in \mathbb{D}}$. Note that by assumption
\begin{equation} \|\mut\|_{\C} \le C_{\eta,K}. \end{equation}

Let $Q \in \sbf$ for some $\sbf \in \mathcal{S}$ and let $\Gamma_{\sbf}$ be the set in $\mathcal{E}$ supplied by the coronization. Let $X_Q \in \Gamma_{\sbf}$ be  such that
$\dist(X_Q,x_Q) < \eta \ell(Q)$, where
lower case  $x_Q$ is the `center' of $Q$ as in \eqref{cube-ball}.
  Let $\Gamma_{\sbf}(Q) := (\Gamma_{\sbf})_{X_Q, C_0 \ell(Q)}$
be the $d$-regular localization of $\Gamma_{\sbf}$ as in Lemma \ref{localregset.lem}, where,
by the triangle inequality and \eqref{eq2.2a},  we can choose $C_0 \gtrsim_{d, C_*} (K + \eta)$ such that  $\Gamma_{\sbf}(Q)$
satisfies
\begin{equation}\label{localstillclose.eq}
\sup_{x\in KQ'} \dist(x,\Gamma_{\sbf}(Q) ) < \eta\,\ell(Q'), \quad \forall Q' \subseteq Q, Q \in \sbf\end{equation}
and
\[\Gamma_{\sbf}(Q) \subset B(x_Q, 5C_0\ell(Q)).\]
Here we use that $\Gamma_{\sbf} \cap B(X_Q, C_0\ell(Q)) \subseteq \Gamma_{\sbf}(Q)$ and the properties of the dyadic cubes. Recall by construction (see Lemma \ref{localregset.lem}) that $\Gamma_{\sbf}(Q_0)$ is closed, $cC_0 \ell(Q) \le \diam(\Gamma_{\sbf}) \le 3C_0 \ell(Q)$ for a constant $c$ depending only on $d$ and $C_{*}$ and that $\Gamma_{\sbf}(Q) \in \Reg(2^{6d}10^dC_{*},10C_0 \ell(Q))$. (Note that we can always use Lemma \ref{regscalechange.lem}, to prove that $\Gamma_{\sbf}(Q) \in \Reg(2^{6d}10^dM^dC_{*},MC_0 \ell(Q))$ for any fixed $M$ and we will do so below.)

By perhaps taking $C_0$ larger (depending on $d$ and $C_*$) we may also assume that if $Q' \in \mathbb{D}(E)$ and $Q^*$ is the grandparent of $Q'$ then $\diam(Q^*) \le C_0 \ell(Q')$. In particular, with this extra condition on $C_0$ and $Q',Q^*$ as above
\[\dist(x_{Q'}, x_{Q^*}) \le C_0 \ell(Q').\]

\subsection{Preliminary observations}
We here record two important observations as lemmas.

\begin{lemma}\label{impob1.lem}
Fix $x \in E$ and $\sbf \in \mathcal{S}$. If there exists an (infinite) nested sequence of cubes
$Q_0 \supsetneq Q_1 \supsetneq Q_2 \dots$, with $x \in Q_k$ and $Q_k \in \sbf$, then $x \in \Gamma_{\sbf}(Q_0)$.
\end{lemma}
\begin{proof} The proof of this lemma is simple. Since $Q_{k+1} \subsetneq Q_k$,
it follows that $\ell(Q_k) \le 2^{-k}\ell(Q_0)$. Then \eqref{localstillclose.eq} gives that $\dist(x,\Gamma_{\sbf}(Q_0)) \le 2^{-k}\ell(Q_0)$ for all $k \in \mathbb{N}$. Since $\Gamma_{\sbf}(Q_0)$ is closed,
$x \in \Gamma_{\sbf}(Q_0)$.
\end{proof}

\begin{lemma}\label{impob2.lem}
If $Q_0 \in \mathbb{D}(E)$, $\mathcal{F}$ is a collection of pairwise disjoint subcubes of $Q_0$ and
    \[\|\mut_\F\|_{\C(Q_0)} \le 1/2\]
    then there exist $\sbf \in \mathcal{S}$ such that $Q \in \sbf$ whenever $Q \in \mathbb{D}_{\mathcal{F},Q_0}$.
\end{lemma}
\begin{proof}
This proof is also simple but requires chasing a few definitions. We first note that we can assume that $\mathcal{F} \neq \{Q_0\}$, as otherwise the lemma is vacuously true. For $Q \in \mathbb{D}_{\mathcal{F},Q_0}$ we have \[\alpha_Q/\mu(Q) \le \mut(Q)/\mu(Q) \le \|\mut_\F\|_{\C(Q_0)}  \le 1/2.\]
By definition $\alpha_Q/\mu(Q) \in \{0,1\}$ and hence $Q \in \mathbb{D}_{\mathcal{F},Q_0}$ can never be a maximal or a bad cube.

Let $\sbf_0$ be the stopping time regime such that $Q_0 \in \sbf_0$. Suppose, for the sake of contradiction, that $Q \in \mathbb{D}_{\mathcal{F},Q_0}$ but $Q \not \in \sbf_0$. Since $Q$ is not maximal or bad, it must be the case that $Q\in \sbf$ for some $\sbf \neq \sbf_0$. It can't be the case that $Q_0 \subseteq Q(\sbf)$ (the maximal cube for $\sbf$) as by coherency of the stopping time regimes, $Q_0 \in \sbf$, which would yield a contradiction. On the other hand, if $Q(\sbf) \subset Q_0$ then since $Q \subset Q(\sbf)$ we have $Q(\sbf) \in \mathbb{D}_{\mathcal{F},Q_0}$. This is contradiction to the fact that $\mathbb{D}_{\mathcal{F},Q_0}$ contains no maximal cubes.
\end{proof}

Combining the two lemmas above, we obtain the following.

\begin{lemma}\label{impobcor.lem}
Let $Q_0 \in \mathbb{D}(E)$ and $\mathcal{F} = \{Q_j\}_{j}$ be a collection of pairwise disjoint subcubes of $Q_0$, with
$\mathcal{F} \neq \{Q_0\}$ and
    \[\|\mut_\F\|_{\C(Q_0)} \le 1/2.\]
    Let $\sbf_0$ be the stopping time regime such that $Q_0 \in \sbf_0$, which exists by Lemma \ref{impob2.lem}.
    Set $A = Q_0 \setminus \cup_{j} Q_j$. If $x \in A$ then $x \in \Gamma_{\sbf_0}(Q_0)$.
\end{lemma}
\begin{proof}
If $x \in A$ then, by properties of dyadic cubes for $Q \in \mathbb{D}_{Q_0}$ with $x \in Q$ we
have that $Q$ is not contained in $\mathbb{D}_{Q_j}$ for any $j$ since this would imply that $x \in Q_j$. Thus $Q \in \mathbb{D}_{\mathcal{F}, Q_0}$ and it follows from Lemma \ref{impob2.lem} that $Q \in \sbf_0$. Then let $Q_i$, $i = 0,1,2 \dots$ be such that $Q_{i + 1}$ is the unique subcube of $Q_i$ such that $x \in Q_{i +1}$. Then $x \in Q_i$ and the collection $Q_i$ satisfy the hypothesis of Lemma \ref{impob1.lem} and hence $x \in \Gamma_{\sbf_0}(Q_0)$.
\end{proof}

\subsection{Proof of Theorem \ref{CoronaBP.thrm} by induction}
We are now ready to prove  Theorem \ref{CoronaBP.thrm} and the proof proceeds via induction. We form two statements.

For $a \ge 0$, let $H(a)$ be the following statement: There exists positive constants $c_a, c'_a, C'_a, \theta_a$ such that if
$\mut(\dd_{Q_0}) < a\mu(Q_0)$, then there exists $F_{Q_0}$ with the following properties:
\begin{enumerate}
\item[(i)] $F_{Q_0} \subset B(x_{Q_0}, 20C_0\ell(Q_0))$ and $\diam(F_{Q_0}) \ge c_a \ell(Q_0)$.
\item[(ii)] $F_{Q_0}\in \Reg(C'_a, C_0\ell(Q))$
 and  $F_{Q_0}$ is in $\BP(\mathcal{E})$ up to the scale $C_0\ell(Q_0)$,
that is, for every $x \in F_{Q_0}$ and $r \in (0, C_0\ell(Q))$ there exists
$\Gamma =\Gamma(x,r)\in \mathcal{E}$ such that
\[\mu(\Gamma \cap F_{Q_0} \cap B(x,r)) \ge \theta_a r^d.\]
\item[(iii)] $\mu(F_{Q_0} \cap Q_0) \ge c'_a \mu(Q_0)$.
\end{enumerate}

Notice that this is a local version of what we want to show, and it is important that we have these {\it localized} regions:
we have no (quantitative) bounds on the number of these constructions that will appear in the next step of the induction procedure below!

We also formulate another hypothesis.  This is the hypothesis we would like to verify (to prove  Theorem \ref{CoronaBP.thrm}) and  we will see that there will only be slight changes in the construction and that $H(a)$ is quite useful in proving $H^*(a+b)$ holds.

For $a \ge 0$ let  $H^*(a)$ be the following statement: There exists positive constants $c_a, c'_a, C'_a, \theta_a$ such that If $\mut(\dd_{Q_0}) < a\mu(Q_0)$ then there exists $F_{Q_0}$ with the following properties:
\begin{enumerate}
\item[(I)] $F_{Q_0}\in \Reg(C'_a)$ and  $F_{Q_0} \in \BP(\mathcal{E})(\theta_a, C'_a)$.
Recall this means for every $x \in F_{Q_0}$ and $r \in (0, \infty)$ there exists $\Gamma \in \mathcal{E}$ such that
\[\mu(\Gamma \cap F_{Q_0} \cap B(x,r)) \ge \theta_a r^d.\]
\item[(II)] $\mu(F_{Q_0} \cap Q_0) \ge c'_a  \mu(Q_0)$.
\end{enumerate}

Using Lemma \ref{cubesenoughbp.lem} we see that to prove the theorem it is enough to verify that $H^*(a)$ holds for all
{\Rd $a \in [0, C_{\eta,K} ]$}, with bounds depending only on $a$ and allowable parameters. In particular, we want to prove that $H^*(C_{\eta,K})$ holds as, by the definition of the coronization, $\mut(\mathbb{D}_Q) \le C_{\eta,K} \mu(Q)$ for {\it all} $Q \in \mathbb{D}$.

We first verify that $H(0)$ and $H^*(0)$ hold, which is essentially trivial. Indeed, $\mut(\mathbb{D}_{Q_0}) = 0$ implies that $\|\mut_\F\|_{\C(Q_0)} = 0 \le 1/2$ with $\mathcal{F} = \emptyset$. Then $A = Q_0$ in Lemma \ref{impobcor.lem} and hence $Q_0 \in \sbf_0$ for some $\sbf_0 \in \mathcal{S}$ with $Q_0 \subset \Gamma_{\sbf}(Q_0)$. To verify $H(0)$ we take $F_{Q_0} =\Gamma_{\sbf_0}(Q_0)$ and property (iii) is satisfied with $c'_a = 1$. By construction $\Gamma_{\sbf_0}(Q_0)$ has all the properties necessary in $H(0)$. To verify $H^*(0)$ we take
$F_{Q_0} = \Gamma_{\sbf_0}$. Since $\Gamma_{\sbf_0}(Q_0) \subset \Gamma_{\sbf_0}$ we have that property (II) is
satisfied with $c'_a = 1$ and the other properties hold trivially.

We now fix $b > 0$ depending on $d$ and $C_*$, such that $Cb \le 1/2$, where $C$ is from Lemma \ref{extraplem.lem}. We prove that if $H(a)$ and $H^*(a)$ hold then $H(a + b)$ and $H^*(a + b)$ hold.

Let $Q_0$ be such that $\mut(\mathbb{D}_{Q_0}) \le (a + b)\mu(Q_0)$. We apply Lemma \ref{extraplem.lem} to $Q_0$ to obtain $\mathcal{F} = \{Q_j\}$, a collection of pairwise disjoint subcubes of $Q_0$, with the properties stated in Lemma \ref{extraplem.lem}. An important observation is that by our choice of $b$ we have
\[\|\mut_\F\|_{\C(Q_0)} \le 1/2,\]
and this allows us to utilize Lemmas \ref{impob2.lem} and \ref{impobcor.lem} (in the case $\mathcal{F} \neq \{Q_0\}$).

We define the following objects\footnote{Here and below for notational convenience we take $\mathcal{B}$ and $\mathcal{G}$ to be collections obtained by using Lemma \ref{extraplem.lem}, and not the collections from the corona decomposition. When we wish to describe a cube from the corona decomposition, we will use the words to `good',`bad' or `maximal' cube.}:
$\mathcal{G} := \mathcal{F} \setminus \mathcal{B}$, where $\mathcal{B}$ is from Lemma \ref{extraplem.lem} and
$G = \cup_{Q_j \in \mathcal{G}} Q_j$. Set $\gamma_a =1 - \frac{a + b}{a + 2b} > 0$ and let $A = Q_0 \setminus (\cup_{Q_j \in \mathcal{F}} Q_j)$. By Lemma \ref{extraplem.lem} we have that
$\mu(B) \le (1- \gamma_a) \mu(Q_0)$ so that
\[\mu(A \cup G) = \mu(Q_0 \cap B^c) \ge \gamma_a \mu(Q_0).\]

This allows us to consider two cases.

\noindent{\bf Case 1}: $\mu(A) > (\gamma_a/2)\mu(Q_0)$. In this case, we use Lemma \ref{impobcor.lem} to say that there exists $\sbf_0$, a stopping time regime, such that $Q_0 \in \sbf_0$, and
$x \in A$ implies that $x \in \Gamma_{\sbf_0}(Q_0)$. To verify $H(a)$ (in Case 1) we again take $F_{Q_0}= \Gamma_{\sbf_0}(Q_0)$ and since $A \subset \Gamma_{\sbf_0}(Q_0)$ we have $\mu(F_{Q_0} \cap Q_0) \ge \gamma_a/2 \mu(Q_0)$ so that property (iii) holds. As the other properties ((i) and (ii)) hold by construction,
this takes care of this case for $H(a)$. To verify $H^*(a)$ (in Case 1) we take $F_{Q_0} = \Gamma_{\sbf_0}$.

\noindent
{\bf Case 2}: $\mu(G) \ge (\gamma_a/2)\mu(Q_0)$. We decompose this case further.

\noindent{\bf Case2a}: $\mathcal{F} = \{Q_0\}$.
 In this case, $\mut(\mathbb{D}_{Q_0} \setminus \{Q_0\}) \le a \mu(Q_0)$ (otherwise $B = Q_0$ which violates the second property in Lemma \ref{extraplem.lem}). By definition of $\mut$ (and pigeon-holing) there exists $Q'_0$ a child of $Q_0$ for which $\mut(Q'_0) \le a\mu(Q'_0)$. In this case we may apply the induction hypothesis\footnote{Here we apply $H(a)$ when trying to prove $H(a + b)$ and $H^*(a)$ when we are trying to prove $H^*(a + b)$ .} to $Q'_0$. Upon allowing the constants to get ``worse'', accommodating for the fact that $\ell(Q'_0) = (1/2) \ell(Q_0)$ we have that $H(a +b)$ and $H^*(a +b)$ hold in this case.

\noindent
{\bf Case2b}:
$\mathcal{F} \neq \{Q_0\}$. By definition of $\mathcal{G}$, for every $Q_j \in \mathcal{G}$ there exists $\widetilde{Q}_j$, a child of $Q_j$, for which it holds that $\mut(\widetilde{Q}_j) \le a\mu(\widetilde{Q}_j)$. Thus,
we can apply the induction hypothesis to any $\widetilde{Q}_j \in \widetilde{\mathcal{G}}$,
the collection of all children of the cubes in $\mathcal{G}$ satisfying $\mut(\widetilde{Q}_j) \le a\mu(\widetilde{Q}_j)$.
Before we do that, we need to work with a collection of separated $\widetilde{Q}_j$ so that later we can maintain the upper regularity when we combine the sets $F_{\widetilde{Q_j}}$ from the induction hypothesis. Using a standard covering lemma argument, we extract from $\widetilde{\mathcal{G}}$ a {\it finite}\footnote{Aside from finiteness, we do not control the cardinality. This is used because we will take a union of closed sets and we need to know this union is also closed.} collection of cubes $\mathcal{G}' = \{Q'_j\} \subseteq \widetilde{\mathcal{G}}$ such that
\[\mu(\cup_{\mathcal{G}'} Q'_j) \gtrsim \mu(\cup_{\mathcal{G}} Q_j) =   \mu(G) \gtrsim (\gamma_a/2)\mu(Q_0),\]
where the implicit constants depend on $d$ and $C_*$,
and
\[\dist(Q'_j, Q'_k) \ge 80C_0\max\{\ell(Q'_j), \ell(Q'_k)\}, \quad \forall Q'_j, Q'_k \in \mathcal{G}' (j \neq k).\]

\begin{remark}\label{wehavegoodparents.rmk} Recall that since $\|\mut_\mathcal{F}\| \le Cb < 1/2$ and $Q_0 \not\in \mathcal{F}$
 it must be the case that $Q_0 \in {\sbf}_0$ for some stopping time regime ${\sbf}_0$ and, in fact, we have, by Lemma \ref{impob2.lem}, that every cube $Q \in \mathbb{D}_{\mathcal{F}, Q_0}$ has the property that $Q \in {\sbf}_0$ as well. In particular, if for every $Q'_j \in \mathcal{G}'$ we write $Q^*_j$ to denote the grandparent of $Q'_j$
we have $Q^*_j \in \mathbb{D}_{\mathcal{F}, Q_0}$ and hence $Q^*_j \in {\sbf}_0$.
\end{remark}

To proceed with {\bf Case2b}  we now construct the set $F_{Q_0}$. If we are proving $H(a+b)$ we set $F_0 := \Gamma_{\sbf_0}(Q_0)$, where ${\sbf}_0$ is the stopping time regime such that $Q_0 \in {\sbf}_0$. If we are proving $H^*(a+b)$, we take $F_0 = \Gamma_{\sbf_0}$ (as usual). Now, regardless of whether we are proving $H(a + b)$ or $H^*(a + b)$, for each $Q'_j \in \mathcal{G}'$, we apply the induction hypothesis $H(a)$ and we let $$F_j := F_{Q'_j}$$ be the set satisfying properties (i), (ii), (iii) (adapted to the scale/size of $Q'_j$). Finally, we set $$F_{Q_0} = \cup_{j =0}^N F_j.$$ Note that $F_{Q_0}$ is a closed set, because it is the {\it finite} union of closed sets. We now verify that the set $F_{Q_0}$ has the necessary properties. When proving $H(a+ b)$ we see that property (i) holds rather trivially, using the triangle inequality\footnote{Here we use that $x_{Q'_j} \in B(x_{Q_0}, C_0 \ell(Q_0))$ so that $B(x_{Q'_j}, 20C_0\ell(Q'_j)) \subset B(x_{Q_0},20C_0\ell(Q'_j) + C_0 \ell(Q_0)) \subset B(x_{Q_0}, 20C_0 \ell(Q_0))$, where we use that $\ell(Q'_j) = (1/2) \ell(Q_j) \le (1/2) \ell(Q_0)$. Thus, using property (i) for $Q'_j$ we have $F_j \subset B(x_{Q_0}, 20C_0 \ell(Q_0))$. }. The next easiest property to verify is (iii) (or (II) when proving $H^*(a+b)$) and we do that next.

{\bf $F_{Q_0}$ has property (iii) (or (II) when proving $H^*(a+b)$)}: Recall that we have shown that $\mu(\cup_{\mathcal{G}'} Q'_j) \gtrsim \mu(Q_0)$, where the implicit constant depends on $d$ and the regularity constant for $E$. Also, by property (iii) for $F_j$ we have that $\mu(F_j \cap Q'_j) \ge c'_a \mu(Q'_j)$.  As the $Q'_j$ are pairwise disjoint it follows that
\[\mu(Q_0 \cap F_{Q_0}) \ge \sum_{j \ge 1} \mu(Q'_j \cap F_j) \ge  \sum_{j \ge 1}  c'_a \mu(Q'_j) \gtrsim \mu(Q_0),\]
which yields property (iii).

It remains to verify property (ii) (or (I) when proving $H^*(a+b)$) for $F_{Q_0}$ and we decompose the verification into two steps: {\bf $F_{Q_0}$ satisfies the upper regularity condition} and {\bf $F_{Q_0}$ satisfies the lower regularity condition and $F_{Q_0} \in \BP(\mathcal{E})$}.

{\bf $F_{Q_0}$ satisfies the upper regularity condition}: Here we prove the upper bound in the definition of $\Reg(C'_a, C_0\ell(Q))$ (when proving $H(a+b)$) or $\Reg(C'_a)$ (when proving $H^*(a+b)$). The proof is the same in either case ($H(a+b)$ or $H^*(a+b)$).

Let $x \in F_{Q_0}$ and $r \in (0, \infty)$. We consider the contributions from $F_0$ and $\{F_j\}_{j \ge 1}$ separately. Notice that $F_0 \subset \Gamma_{\sbf_0}$ regardless of whether we are showing $H(a+b)$ or $H^*(a+b)$, so that $F_0$
satisfies the upper regularity condition.
If $B(x,r)$ meets $F_0$ then $B(x,r) \subset B(y,2r)$ for some $y \in F_0$ and hence $\mu(B(x,r) \cap F_0) \le \mu(B(y,2r) \cap F_0) \lesssim r^n$ by the upper regularity property for $F_0$. We dominate the contribution from the union of the sets $\{F_j\}_{j \ge 1}$ by
\begin{align*}\sum_{j \ge 1} \mu(B(x,r) \cap F_j) &\le \sum_{j: \ell(Q'_j) > r} \mu(B(x,r) \cap F_j)   + \sum_{j: \ell(Q'_j) \le r} \mu(B(x,r) \cap F_j)\\
& =: T_1 + T_2.
\end{align*}

We first handle the term $T_1$ and we will see there is at most one non-zero term in $T_1$. Note that by hypothesis $F_j \subset B(x_{Q'_j}, 20C_0\ell(Q'_j))$ so that (as $C_0 > 1$) if $B(x,r)$ meets $F_j$ with $\ell(Q'_j) > r$, then $B(x,r) \subset B(x_{Q'_j}, 21C_0\ell(Q'_j))$. If $B(x,r)$ were to meet $F_k$ for $k \neq j$ then $$B(x_{Q'_j}, 21C_0\ell(Q'_j)) \cap B(x_{Q'_k}, 21C_0\ell(Q'_k)) \neq \emptyset,$$ and this contradicts the fact that $\dist(Q'_j, Q'_k) \ge 80C_0 \max\{\ell(Q'_j), \ell(Q'_k)\}$. Thus, $T_1 = 0$ or $T_1 = \mu(B(x,r) \cap F_j)$ for a single $F_j$ with $\ell(Q'_j) > r$ and, in that case, we can use the upper regularity of $F_j$ to conclude that $T_1 \lesssim r^d$.

We next handle the term $T_2$.  The collection of $F_j$ with $\ell(Q'_j) \le r$ is contained in $B(x, Cr)$ ($C = 80C_0$ will do). Using the upper regularity for $F_j$ we have $\mu(F_j) \lesssim \diam(F_j)^d \lesssim \ell(Q'_j)^d \lesssim \mu(Q'_j)$ and using that $F_j \subset B(x_{Q'_j}, 20C_0\ell(Q'_j))$ we have  $Q'_j \subset B(x,2Cr)$. Thus
\begin{align*}T_2 &= \sum_{j: \ell(Q_j') \le r} \mu(F_j \cap B(x,r))\\
& \lesssim \sum_{j: Q'_j \subset B(x,2Cr)} \mu(Q'_j) \le \mu(B(x,Cr) \cap E) \lesssim r^d,\end{align*}
where we have used that $E \in \Reg(C_*)$. This shows that $F_{Q_0}$ satisfies the upper regularity condition.

{\bf $F_{Q_0}$ satisfies the lower regularity condition and $F_{Q_0} \in \BP(\mathcal{E})$}:
Let $x \in F_{Q_0}$ and $r >0$ with the further restriction that $r < C_0\ell(Q)$ in the case we are proving $H(a+b)$. We decompose the proof into cases:
\begin{align*}
&\mbox{{\bf Case $\alpha$}: $x \in F_0$.}\\
&\mbox{{\bf Case $\beta$}: $x \in F_j$ for some $j \ge 1$ and $r < (800 + \eta)C_0 \ell(Q'_j)$.}\\
&\mbox{{\bf Case $\gamma$}: $x \in F_j$ for some $j \ge 1$ and $r \geq  (800 + \eta)C_0 \ell(Q'_j)$.}
\end{align*}

{\bf Case $\alpha$}: $x \in F_0$. In this case, if we are trying to prove $H(a+b)$ we
just use that $F_0 \in \Reg(2^{6d}C_{*}, C_0\ell(Q)$ and $F_0 \subset \Gamma_{\sbf_0}$ with $\Gamma_{\sbf_0} \in \mathcal{E}$.
Thus, by construction $F_0$ satisfies the lower regularity condition and is $\BP(\mathcal{E})$ (up to scale $\ell(Q_0)$).
In particular, $\mu(B(x,r) \cap F_{Q_0}) \ge \mu(B(x,r) \cap F_{0})   \gtrsim r^d$ and the set $B(x,r) \cap F_0$ is a subset of a set in $\mathcal{E}$. This takes care of case $\alpha$ when showing $H(a+b)$. When showing $H^*(a +b)$, the proof is almost identical.

{\bf Case $\beta$}: $x \in F_j$ for some $j \ge 1$ and $r < (800 + \eta)C_0 \ell(Q'_j)$. In this case, we simply use
the $\BP(\mathcal{E})$ and lower regularity conditions for $F_j$. In particular,  $$\mu(B(x,r) \cap F_{Q_0}) \ge \mu(B(x,r) \cap F_{j}) \gtrsim r^n$$ by the lower regularity property of $F_j$ and
using the $\BP(\mathcal{E})$ property of $F_j$ there exists $\Gamma \in \mathcal{E}$ such that
$\mu(B(x,r) \cap F_{Q_0} \cap \Gamma) \ge \mu(B(x,r) \cap F_{j} \cap \Gamma) \gtrsim r^n$.
(Recall that these properties hold for $F_j$ up to the scale $\ell(Q_j)$ by property (ii) for $F_j$.)
This takes care of case $\beta$.

{\bf Case $\gamma$}: $x \in F_j$ for some $j \ge 1$ and $r \geq (800 + \eta)C_0\ell(Q'_j)$. Recall by the discussion above (see Remark \ref{wehavegoodparents.rmk}) that $Q^*_j$ the grandparent of $Q'_j$ is in the stopping time regime ${\sbf}_0$ and hence, by choice of $C_0$, $\dist(x_{Q^*_j}, \Gamma_{\sbf_0}(Q_0)) \le  \eta \ell(Q_j^*)$. Moreover, by the choice of $C_0$ we have that $\dist(x_{Q'_j},x_{Q^*_j}) \le C_0\ell(Q_j)$ and $F_j \subset B(x_{Q'_j}, 40C_0 \ell(Q'_j))$. Thus, there exists $z \in \Gamma_{\sbf_0}(Q_0)$ such that $\dist(x_{Q'_j}, z) < \eta \ell(Q'_j) + 80C_0 \ell(Q'_j)$ and the condition on $r$ shows that $B(z,r/2) \subset B(x,r)$. Now, using the arguments of case $\alpha$
with $r$ replaced by $r/2$ (which produces slightly worse estimates) we can conclude that case $\gamma$ can be taken care of

We have now proved property (ii) if we were trying to prove $H(a+b)$ (or (I) when proving $H^*(a+b)$) and hence $H(a + b)$ and $H^*(a+b)$ hold.

\section{Applications to parabolic uniform rectifiability}\label{parabolic.sect}

In this section we use Theorem \ref{CoronaBP.thrm}, Theorem \ref{CoronaGLtrans.thrm} and results from \cite{BHHLN} to give a new characterization of parabolic uniformly rectifiable sets. We consider $n \in \mathbb{Z}$ and we will always assume that $n\geq 1$. We consider the  Euclidean $ (n+1) $-space
$ \mathbb R^{n+1} $ where points will be denoted by $ (X,t) = ( x_1,
 \dots,  x_n,t)$,  where $ X = ( x_1, \dots,
x_{n} ) \in \mathbb R^{n } $ and $t$ represents the time-coordinate. Let $  \langle \cdot ,  \cdot  \rangle $  denotes  the standard inner
product on $ \mathbb R^{n} $ and let  $  | X | = \langle X, X \rangle^{1/2} $ be
the  Euclidean norm of $ X. $  We let $||(X,t)||:=|X|+|t|^{1/2}$. Given $(X,t), (Y,s)\in\mathbb R^{n+1}$ we let $$d_p(X,t,Y,s)=d_p((X,t),(Y,s))=|X-Y|+|t-s|^{1/2}.$$ Throughout the section we consider the triple
$$(X,\dist,\mu):\ X:=\mathbb R^{n+1},\ \dist:=d_p,\ \mu:=H^{n-1}\times H^1,$$
and we let
$$d:=n+1.$$
We let $\mathcal{P}$ denote the set of hyperplanes in $X$ containing a line
parallel to the $ t $ axis.

As noted in the introduction, in \cite{HLN}, \cite{HLN1} the third and fifth author, together with John Lewis,  introduced a  notion of parabolic uniformly rectifiable sets. Using the notation introduced the paper this notion of parabolic uniformly rectifiable sets can be defined as follows.

\begin{definition}\label{p-UR} Let $E\subset X$.  $E$ is parabolic uniformly rectifiable set, or
$E$ is uniformly rectifiable in the parabolic sense, if
$E \in \Reg(C) \cap \GLem (\mathcal{P},2,2)$,  for some finite $C\geq 1$.
\end{definition}

To introduce the Lipschitz type graphs of which we want to consider big pieces,  we first note that in \cite{H}, \cite{HL}, \cite{LM}, \cite{LS}, \cite{M} the authors established the correct notion of (time-dependent) regular parabolic Lipschitz graphs from the point of view of parabolic singular integrals and parabolic measure. To expand a bit on this, recall that  $\psi:\mathbb R^{n-1}\times\mathbb R\to \mathbb R$ is called Lip(1,1/2), or Lip(1,1/2)  regular, with constant $b$, if
\begin{eqnarray}\label{1.1}
|\psi(x,t)-\psi(y,s)|\leq b(|x-y|+|t-s|^{1/2})\
\end{eqnarray}
whenever $(x,t)\in\mathbb R^{n}$, $(y,s)\in\mathbb R^{n}$.  $\Gamma=\Gamma_\psi\subset\mathbb R^{n+1}$ is said to be a (unbounded) Lip(1,1/2)  graph, with constant $b$, if
\begin{eqnarray}\label{1.1a}
\Gamma=\Gamma_\psi=\{(x,x_n,t)\in\mathbb R^{n-1}\times\mathbb R\times\mathbb R:x_n=\psi(x,t)\}
\end{eqnarray} for some Lip(1,1/2)  function $\psi$ having Lip(1,1/2)  constant bounded by $b$.   $ \psi = \psi ( x, t ) : \mathbb R^{n-1}\times\mathbb R\to \mathbb R$ is called a regular parabolic Lip(1,1/2)  function
with parameters $b_1$ and $b_2$, if $\psi$ has compact support
and satisfies   \begin{eqnarray} \label{1.7}
(i)&&|\psi(x,t)-\psi(y,t)|\leq b_{1}|x-y|, x, y\in \mathbb R^{n-1}, t\in\mathbb R,\notag\\
(ii)&& D_{1/2}^t\psi\in BMO(\mathbb R^n), \ \ \|D_{1/2}^t\psi\|_*\leq b_2<\infty.
\end{eqnarray}
 Here   $ D_{1/2}^t \psi  (x, t) $ denotes
the $ 1/2 $ derivative in $ t $ of $ \psi ( x, \cdot ), x $ fixed.
This half derivative in time can be defined by way of the Fourier
transform or by
\begin{eqnarray} \label{1.8}
 D_{1/2}^t  \psi (x, t)  \equiv \hat c \int_{ \mathbb R }
\, \frac{ \psi ( x, s ) - \psi ( x, t ) }{ | s - t |^{3/2} } \, ds
\end{eqnarray} for properly chosen $ \hat c$. $ \| \cdot \|_* $ denotes the
norm in parabolic $BMO(\mathbb R^{n})$.  For a definition of the space parabolic $BMO(\mathbb R^{n})$ we refer to \cite{HLN}. It is well known, see \cite{HL}, that if $\psi$
is a regular parabolic Lip(1,1/2)  function
with parameters $b_1$ and $b_2$, then $\psi$ is Lip(1,1/2)  regular with constant $b=b(b_1,b_2)$. However, there are examples of
functions $\psi$ which are  Lip(1,1/2)  regular but not regular parabolic Lip(1,1/2), see \cite{LS}.

\begin{definition}\label{defgpg}
We say that $\Gamma$ is a regular (or good) parabolic graph with parameters $b_1$ and $b_2$, $\Gamma \in \GPG(b_1,b_2)$ for short or simply $\Gamma \in \GPG$ if
 the parameters are implicit, if after a possible rotation of the spatial variables $\Gamma=\Gamma_\psi$ can be represented as in \eqref{1.1a}
for some a regular parabolic Lip(1,1/2)  function $ \psi = \psi ( x, t ) : \mathbb R^{n-1}\times\mathbb R\to \mathbb R$
with parameters $b_1$ and $b_2$.
\end{definition}

We next formulate the following lemma which states that good parabolic graphs are uniformly rectifiable in the parabolic sense.
\begin{lemma}\label{glem4goodgraph.thrm} Assume that after a possible rotation of the spatial variables $\Gamma=\Gamma_\psi$ can be represented as in \eqref{1.1a}
for a Lip(1,1/2) function $\psi$. Then  $\Gamma$ is uniformly rectifiable in the parabolic sense if and only if $ D_{1/2}^t\psi\in BMO(\mathbb R^n)$, in particular, if
$\Gamma \in \GPG(b_1,b_2)$ then
$\Gamma \in \GLem (\mathcal{P}, 2,2, M)$, where $M$ only depends on $d,b_1,b_2$.
\end{lemma}
\begin{proof} The fact that
$\Gamma \in \GPG(b_1,b_2)$ implies  $\Gamma \in \GLem (\mathcal{P}, 2,2, M)$, with $M$
depending on $d,b_1,b_2$, is proved in \cite[pp. 249-251]{H}. In \cite{H} a different formulation
of the half-order derivative condition was used, but the two formulations are equivalent
for Lip(1,1/2) graphs, as is proved in \cite[Section 7]{HL}. The converse implication, i.e., that if  $\Gamma$ is uniformly rectifiable then
$ D_{1/2}^t\psi\in BMO(\mathbb R^n)$, is proved in
\cite[pp. 370-373]{HLN}.  For both of these implications, the proofs in the cited references are given in terms of
continuous parameter versions of $\beta$, rather than the dyadic version, but it is easy to see that a Carleson
measure condition for the former is equivalent to a dyadic Carleson measure condition for the latter.
\end{proof}

To proceed we next state the following result concerning Corona decompositions of parabolic uniformly rectifiable sets. The theorem  is proved in \cite{BHHLN}.

\begin{theorem}\label{Corona}
Let $E\subset X$ and assume that $E \in \Reg(C_*) \cap \GLem (\mathcal{P},2,2,M)$,   i.e., $E$ is uniformly rectifiable in the parabolic sense. Suppose
$0 < \eta$
and $K>  1$. Then there exists $b_1,b_2$, both depending at most of $d,C_*,M$ and $\eta,K$ such that $E$ admits an $(\eta, K)$-coronization with respect to $\mathcal{E} = \GPG(b_1,b_2)$.
\end{theorem}

Using Theorem \ref{Corona} we are able to specify and apply Theorem \ref{CoronaBP.thrm} and Theorem \ref{CoronaGLtrans.thrm}  to the parabolic setting. Theorem \ref{CoronaBP.thrm} yields the following.

\begin{theorem}\label{BP2parabolic.thrm}
Suppose that $E \in \Reg(C_*)$ admits a coronization with respect to $\mathcal{E} = \GPG(b_1,b_2)$ for some fixed $b_1,b_2$. Then $E \in \BP^2(\GPG(b_1,b_2))$ with constants depending on $d, C_*,b_1,b_2$ and the constants in the coronization.
\end{theorem}

Theorem \ref{glem4goodgraph.thrm} says that
$\Gamma \in \GPG(b_1,b_2)$ implies $\Gamma \in \GLem (\mathcal{P},2,2,M)$ with $M = M(b_1,b_2, d)$. It is easy to deduce the weak geometric lemma from
the geometric lemma (see, e.g.,
\cite[Section 2]{HLN}),
and in particular,  $\Gamma \in \GPG(b_1,b_2)$ implies that for every $\epsilon \in (0,1]$ there exists $M_\epsilon$ depending on $d, b_1,b_2$ and $\epsilon$ such that $\Gamma \in \WGLem (\mathcal{P},\epsilon, M_\epsilon)$.
Due to the graph structure the following lemma holds.

\begin{lemma}\label{graphbwgl} Assume that after a possible rotation of the spatial variables $\Gamma=\Gamma_\psi$ can be represented as in \eqref{1.1a}
for a Lip(1,1/2) function $\psi$. Suppose that  for every $\epsilon \in  (0,1]$ there exists $M_\epsilon$ such that $\Gamma \in \WGLem (\mathcal{P},\epsilon, M_\epsilon)$. Then for every $\epsilon \in (0,1]$ there exists $M'_\epsilon$ depending on the Lip(1,1/2) constant, $d$ and the function $\gamma(\epsilon):= M_\epsilon$ such that $\Gamma \in \BWGLem (\mathcal{P},\epsilon, M'_\epsilon)$. In particular, if $\Gamma \in \GPG(b_1,b_2)$ then for every $\epsilon \in (0,1]$ there exists $M'_\epsilon$ depending on $d, b_1,b_2$ and $\epsilon$ such that $\Gamma \in \BWGLem (\mathcal{P},\epsilon, M'_\epsilon)$.
\end{lemma}
\begin{proof} Given $\epsilon \in  (0,1]$ assume that $\Gamma \in \WGLem (\mathcal{P},\epsilon, M_\epsilon)$. Let $Q\in \mathbb D(\Gamma)$ and let $(X_Q,t_Q)$ denote the center of $Q$. Consider
\[\beta_{\mathcal{P},\Gamma}(Q) := \inf_{P \in \mathcal{P}}\left\{ \diam(Q)^{-1} \sup_{(Y,s) \in 2Q} \dist((Y,s),P) \right\}.\]
Given $Q$ we let $$\tilde {\mathcal{P}}_Q:=\{P\in  \mathcal{P}:\ P \mbox{ passes through } (X_Q,t_Q)\}$$
and we introduce
\[\tilde \beta_{\mathcal{P},\Gamma}(Q) :=\inf_{P \in \tilde {\mathcal{P}}_Q}\left\{ \diam(Q)^{-1} \sup_{(Y,s) \in 2Q} \dist((Y,s),P) \right\}.\]
Then
\begin{align}\label{equiv}
\beta_{\mathcal{P},\Gamma}(Q)\leq \tilde \beta_{\mathcal{P},\Gamma}(Q)\leq c  \beta_{\mathcal{P},\Gamma}(Q)
\end{align}
where the constant $c\geq 2$ is independent of $Q$ and $\Gamma$. Consider an arbitrary $P\in \tilde {\mathcal{P}}$. Then
\begin{align}\label{equiv+} b\beta_{\mathcal{P},\Gamma}(Q)&\leq
\diam(Q)^{-1} \sup_{(Y,s) \in 2Q} \dist((Y,s),P)\notag\\
&+\diam(Q)^{-1} \sup_{ (Z,\tau) \in P \cap B((X_Q,t_Q),2\diam(Q))}  \dist((Z,\tau),\Gamma).
\end{align}
Using that $\Gamma$ is the graph of a Lip(1,1/2) function with constant $b$ we see that there exists a constant $K=K(b)\geq 1$ such that
\begin{equation}\label{equiv++}
 \sup_{ (Z,\tau) \in P \cap B((X_Q,t_Q),2\diam(Q))}  \dist((Z,\tau),\Gamma)\leq  \sup_{(Y,s) \in KQ} \dist((Y,s),P).
\end{equation}
Combining \eqref{equiv}-\eqref{equiv++} we deduce that
\begin{equation}\label{equiv+++} b\beta_{\mathcal{P},\Gamma}(Q)\leq c(b)\beta_{\mathcal{P},\Gamma}(KQ).
\end{equation}
Using this inequality we see that for every $\epsilon \in (0,1]$ there exists $M'_\epsilon$ depending only on $b$, $d$ and the function $\gamma(\epsilon):= M_\epsilon$ such
that $$\Gamma \in \BWGLem (\mathcal{P},\epsilon, M'_\epsilon).$$
\end{proof}

Theorem \ref{CoronaGLtrans.thrm} now yields the following where Lemma \ref{graphbwgl} is used in part (iii).

\begin{theorem}\label{t4.6}
Suppose that $E \in \Reg(C_*)$ admits a coronization with respect to $\mathcal{E} = \GPG(b_1,b_2)$ for some fixed $b_1,b_2$. Then the following hold.
\begin{enumerate}
    \item [(i)] $E \in \GLem (\mathcal{P}, 2,2, M)$ where $M$ depends on $C_*,b_1,b_2, d$ and the constants in the coronization. In particular, $E$ is uniformly rectifiable in the parabolic sense.
    \item [(ii)] For every $\epsilon \in (0,1]$ there exists $M_\epsilon$, depending on $\epsilon, C_*,b_1,b_2, d$ and the constants in the coronization, such that $E \in \WGLem (\mathcal{P},\epsilon, M_\epsilon)$.
    \item [(iii)] For every $\epsilon \in (0,1]$ there exists $M'_\epsilon$, depending on $\epsilon, C_*,b_1,b_2, d$ and the constants in the coronization, such that $E \in \BWGLem (\mathcal{P}, \epsilon, M'_\epsilon)$.
\end{enumerate}
\end{theorem}

Specializing to the case of the Geometric Lemma, we conclude with the following, which is an immediate corollary of previously stated results.
\begin{theorem}\label{t4.16} Suppose that $E\subset \ree$, and $E \in \Reg(C_*)$, for some $C_*\geq 1$.
Then the following are equivalent:
\begin{enumerate}
    \item [(i)] $E$ is uniformly rectifiable in the parabolic sense.
      \item [(ii)] $E$ admits a coronization with respect to $\mathcal{E} = \GPG(b_1,b_2)$ for some $b_1,b_2$.
     \item [(iii)]  $E\in BP^2(\GPG(b_1,b_2))$.
\end{enumerate}
\end{theorem}
Indeed, observe that (i) implies (ii) is Theorem \ref{Corona},  which is one of the results that will appear in our 
forthcoming paper \cite{BHHLN}; (ii) implies (iii) is Theorem \ref{BP2parabolic.thrm}; finally, the implication
(iii) implies (i) follows from Lemma \ref{glem4goodgraph.thrm},
and the stability result Proposition \ref{geolemprop1.prop}, applied with $p=q=2$, and with $\mathcal{A}=\mathcal{P}$,
the collection of all
hyperplanes parallel to the $t$-axis.

We conclude this section with a corollary concerning singular integral operators on parabolic uniformly rectifiable sets.
Let $d=n+1$ denote the parabolic homogeneous dimension of $\rn$.
Given a positive integer $N$, we shall say that a singular kernel $K$ satisfies ($d$-dimensional)
``$C$-$Z(N)$" estimates if $K \in C^N(\ree\setminus \{0\})$ and
\[
|\nabla_X^j\,\partial_t^kK(X,t)| \,\leq\, C_{j,k}\, \|(X,t)\|^{-d-j-2k} \,, \quad \forall \, 0\leq j+k\leq N\,.
\]

\begin{corollary}\label{c4.17} Suppose that
 $E\subset \ree$ is uniformly rectifiable in the parabolic sense.  Let 
 $K\in C^N(\ree\setminus \{0\})$ be a $C$-$Z(N)$ kernel.
Assume further that $K(X,t)$ is odd in $X$, for each fixed $t$, i.e., $K(X,t) = - K(-X,t)$,
 for every $X\in \rn, t\in \re$.  For each $\eps>0$, define the truncated SIO
 \[T_\eps f(X,t):= \iint_{\|(X-Y,t-s)\|>\eps} K(X-Y,t-s)\, f(Y,s)\, dY ds\,.\]
Then for $N$ sufficiently large, we have the uniform $L^2$ bound
 \begin{equation}\label{eq4.18}
 \sup_{\eps>0} \|T_\eps f\|_{L^2(E)} \leq\, C\|f\|_{L^2(E)}\,,
 \end{equation}
 where $C$ depends only on $K$, $n$, and the constants in the parabolic uniformly rectifiable and ADR conditions
 for $E$.
 \end{corollary}

By the method of Guy David (and Cotlar's inequality for maximal singular integrals, see \cite[Proposition III.3.2]{D4}), 
the $L^2$ bounds in \eqref{eq4.18} are stable under the ``big pieces functor".  Thus, by Theorem \ref{t4.16}, the conclusion of the corollary is reduced to the case that $E$ is a Good Parabolic Graph in the sense of Definition \ref{defgpg}.
In turn, the latter case follows essentially from  \cite{H} (using the method of \cite{CDM}).  
The results in \cite{H} apply directly only to the case that $K$ satisfies the parabolic homogeneity condition
 \[K(\rho X,\rho^2 t) = \rho^{-d} K(X,t)\,,\quad \forall\, \rho >0\,,\]
 but in fact the arguments in \cite{H} may be adapted to treat the non-homogeneous case as well.
 Details will appear in \cite{BHHLN}.

To conclude this section we make an observation that draws a contrast between the study of uniformly rectifiable sets and sets which are uniformly rectifiable in the {\it parabolic} sense.

\begin{observation}
Let $n \ge 2$. There exists a Lip(1,1/2) graph $\Gamma$ in $\re^{n+1}$ such that for every $\epsilon > 0$, $\Gamma \in  \BWGLem (\mathcal{P}, \epsilon, M'_\epsilon)$ with $M'_\epsilon$ depending on $\epsilon$, but $\Gamma$ is not uniformly rectifiable in the parabolic sense. 
\end{observation}
This observation seems to suggest that there may be no `useful' Carleson {\it set} conditions which characterize sets which are uniformly rectifiable in the {\it parabolic} sense (there are many such characterizations of uniformly rectifiable sets). To make the observation, we simply use the Lip(1,1/2) graph constructed at the end of \cite{HLN}, which is not uniformly rectifiable in the parabolic sense (and relies on the work of Lewis and Silver \cite{LS}). The function defining the graph is a product of a smooth compactly supported function of the spatial variables and a function of $t$.
The modulus of continuity of the function of $t$ is bounded by $\omega(\tau) = C\min\{(\tau/\log(1/\tau))^{1/2}, 1\}$. 
From this information we make the estimate 
\[b\beta_{\mathcal{P}}(Q) \le C_n \min\{(\log(1/\ell(Q))^{-1/2}, \ell(Q)^{-1}\}, \quad \forall Q \in \mathbb{D}.\]
Therefore if $b\beta_{\mathcal{P}}(Q) > \epsilon > 0$ then it must be the case that $\ell(Q) \in [e^{-(C_n/\epsilon)^2}, C_n/\epsilon]$. Thus, for any fixed $\epsilon > 0$, the collection of cubes $Q \in \mathbb{D}$ such that $b\beta_{\mathcal{P}}(Q) > \epsilon$ is a subset of a \underline{finite} collection of \underline{generations} of the dyadic lattice $\mathbb{D}$. Using this fact, for any fixed $\epsilon > 0$ the packing condition for the collection $\{Q \in \mathbb{D}: b\beta_{\mathcal{P}}(Q) > \epsilon\}$ in Definition \ref{BWGL.def} holds. 
\appendix
\section{Proofs of  Propositions \ref{geolemprop1.prop}-\ref{geolemprop3.prop}}\label{geolem.app}

We here prove Propositions \ref{geolemprop1.prop}, \ref{geolemprop2.prop} and \ref{geolemprop3.prop} concerning the stability of
various `geometric lemmas' in this general setting. As previously stated, we claim modest originality  but the propositions are used in our application to parabolic uniform rectifiability in the previous sections and the propositions have previously not occurred in the literature in that context. While the proofs of the propositions follow almost exactly those
in \cite{DS2,Rigot}, a difference is that we work with the dyadic versions of the $\beta$'s and $I$'s, rather than
the continuous parameters versions.  In the generality in which we work here, it is not clear to us whether
the continuous parameter $\beta$'s are necessarily measurable.

As the proofs of Propositions \ref{geolemprop1.prop}, \ref{geolemprop2.prop} and \ref{geolemprop3.prop}  all start the same we start out by proving them simultaneously. Recall that we are assuming $E \in \BP(\mathcal{E})$ where $\mathcal{E}$ is a collection of
regular sets satisfying a particular geometric lemma (depending on which proposition we are proving).

Fix $R \in \mathbb{D}(E)$ and let $\widetilde{E} \in \mathcal{E}$ be such that $\mu(\widetilde{E} \cap R) \ge c\theta$ (see Lemma \ref{cubesenoughbp.lem}). Suppose that $Q \subseteq R$ and $Q \cap \widetilde{E} \neq \emptyset$. Then there exists $\widetilde{Q} =\widetilde{Q}(Q) \in  \dd(\widetilde{E})$ such that $\diam(\widetilde{Q}) \ge 10 \diam(Q)$, $Q \cap \tilde{Q} \neq \emptyset$, and
\[\diam(\widetilde{Q}) \le C_2\diam(Q),\]
for some constant $C_2$ depending only on the $d$-regularity constant and dimension. For every such cube $Q$ we choose one such $\widetilde{Q}$. Note that by regularity of $E$ for fixed $Q' \in \mathbb{D}(\widetilde{E})$,
\begin{equation}\label{cantbechosentoomany.eq} \# \{Q \in \mathbb{D}(E): 10 \diam(Q)\le \diam(Q') \le C_2 \diam(Q) \} < L,
\end{equation}
where $L$ depends only on the $d$-regularity of $E$ and dimension.  Let $\tilde{\delta}(y) = \dist(y, \widetilde{E})$ and $E_1 = E \setminus \widetilde{E}$.

 In the following we will for simplicity write ${\beta}_q, {\beta}, b{\beta}$ for ${\beta}_{\mathcal{A},q}, {\beta}_{\mathcal{A}}, b{\beta}_{\mathcal{A}}$. In analogy we let $\tilde{\beta}_q, \tilde{\beta}, b\tilde{\beta}$ denote the $\beta$'s defined with respect to $\tilde{E}$

We next state five lemmas, Lemmas \ref{ds1.12JN.eq}-\ref{DSlem1.37}.  Lemma \ref{geoclaim1.cl}, Lemma \ref{geoclaim2.cl} and  Lemma \ref{geoclaim3.cl} pertain to  Propositions \ref{geolemprop1.prop}, \ref{geolemprop2.prop} and \ref{geolemprop3.prop}, respectively. We postpone the proofs of  Lemma \ref{ds1.12JN.eq}-\ref{DSlem1.37} for now to completed the proof of  Propositions \ref{geolemprop1.prop}, \ref{geolemprop2.prop} and \ref{geolemprop3.prop}. The proofs of the lemmas given at the end of the section.

\begin{lemma}[{\cite[Lemma IV.1.12]{DS2}}]\label{ds1.12JN.eq}
Let $E \in \Reg(C_*)$ and $\alpha: \dd(E) \to [0,\infty)$. Suppose that there exists $N > 0$ and $\eta > 0$ such that
\[\mu\left(\left\{x \in R: \sum_{\substack{Q \ni x,\ Q \subseteq R}}\alpha(Q) \le N \right\}\right) \ge \eta \mu(R), \quad \forall R \in \mathbb{D}.\]
Then there exists $C= C(C_*,d, N, \eta)$ such that
\[\sum_{Q \subseteq R} \alpha(Q)\mu(Q) \le C\mu(R),\quad \forall R \in \mathbb{D}.\]
\end{lemma}

\begin{lemma}\label{geoclaim1.cl} For $q \in (0,\infty)$ if $Q \cap \widetilde{E} \neq \emptyset$
\[\beta_q(Q) \le C(q)\left[ \tilde{\beta}_q(\widetilde{Q}) + {I}_q(Q) \right],\]
where
\[{I}_q(Q) = \left(\mu(Q)^{-1} \int_{\substack{2Q \cap E_1 \\\tilde{\delta}(y) < 2\diam(Q)}}  [\tilde{\delta}(y)(\diam(Q))^{-1}]^q \, d\mu(y) \right)^{1/q}. \]
\end{lemma}

\begin{lemma}\label{geoclaim2.cl} If $Q \cap \widetilde{E} \neq \emptyset$
\[\beta(Q) \le C_u\left[ \tilde{\beta}(\widetilde{Q}) + I_\infty(Q) \right],\]
where
\[{I}_\infty(Q) = \sup_{ \substack{y \in 2Q \cap E_1 \\ \tilde{\delta}(y) < 2\diam(Q) }}[\tilde{\delta}(y)(\diam(Q))^{-1}]. \]
\end{lemma}

\begin{lemma}\label{geoclaim3.cl}
If $Q \cap \widetilde{E} \neq \emptyset$
\[b\beta(Q) \le C_b\left[ b\tilde{\beta}_q(\widetilde{Q}) + I_\infty(Q) + \widetilde{I}_\infty(\widetilde{Q}) \right],\]
where
\[\widetilde{I}_\infty(\widetilde{Q}) = \sup_{\substack{z \in 2\widetilde{Q} \cap E_1 \\ \dist(z, E) < 2\diam(\widetilde{Q})}} [\dist(z, E)(\diam(Q))^{-1}]. \]
\end{lemma}

\begin{lemma}[{\cite[Lemma IV.1.37]{DS2}}]\label{DSlem1.37}
Consider $E_2, E_3 \in \Reg(C_*)$ and let, for $Q \in \mathbb{D}(E_2)$ and $q \in (0,\infty)$,
\[I_q(Q)= \left(\mu(Q)^{-1}\int_{\substack{z \in 2Q \\ \dist(z, E_3) < 2\diam(Q)}} [\dist(z, E_3)(\diam(Q))^{-1}]^{q} \, d\mu(z) \right)^{1/q}\]
and set
\[I_\infty(Q)= \sup_{\substack{z \in 2Q \\ \dist(z, E_3) < 2\diam(Q)}} [\dist(z, E_3)(\diam(Q))^{-1}].\]
Assume that $p \in (0,\infty)$ and $q \in (0,\infty]$ satisfy $\tfrac{1}{q} - \tfrac{1}{p} + \tfrac{1}{d} > 0$. Then there exists a constant $C_{p,q}$ depending on $p,q, d$ and $C_*$ such that
\begin{equation}\label{Icarl.eq}\sum_{Q \subseteq R} I_q(Q)^p\mu(Q) \le C_{p,q} \mu (R), \quad \forall R \in \mathbb{D}.\end{equation}
Moreover,
\[\sum_{\substack{Q \subseteq R \\ I_\infty(Q) > \epsilon}} \mu(Q) \le C\mu(R), \]
where $C$ depends only on $d$ and $C_*$.
\end{lemma}

\subsection{Completing the  proofs of the main propositions}
Now with Lemmas \ref{geoclaim1.cl}, \ref{geoclaim2.cl} and \ref{geoclaim3.cl} in hand, we prove each proposition separately. To prove Proposition \ref{geolemprop2.prop} is similar and easier than proving Proposition \ref{geolemprop3.prop} and we therefore leave the proof of  Proposition \ref{geolemprop2.prop} to the interested reader. In the following we prove Proposition \ref{geolemprop1.prop} and Proposition \ref{geolemprop3.prop} and we start with the proof of the latter.

Notice for $x \in R \cap \widetilde{E}$ when $x \in Q \subseteq R$ the cube we associate in $\widetilde{E}$ satisfies $\widetilde{Q} \in B^*(R) = B(x_R, 10C_2\diam(Q))$ and $\diam(\widetilde{Q}) \le C_2\diam(R)$. By making $C_2$ larger we may assume that $\widetilde{Q}$ is contained in a cube $\widetilde{R}$ such that $C_2 \diam(R) \le \diam(\widetilde{R}) \le C_2^2 \diam(R)$. Set
\[\mathcal{F} = \{\widetilde{R} \in \mathbb{D}(\widetilde{E}):\widetilde{R} \cap B^*(R),C_2 \diam(R) \le \diam(\widetilde{R}) \le C_2^2 \diam(R)\}.\]
By the $d$-regularity of $\widetilde{E}$ it follows that $\# \mathcal{F} \le L'$, where $L'$ depends only on dimension and the $d$-regularity of $\widetilde{E}$. Moreover, by the hypothesis that $\widetilde{E} \in \BWGLem (\mathcal{A}, \epsilon, M_\epsilon)$ and Lemma \ref{DSlem1.37}
\begin{align}\label{tildepackp3.eq} \sum_{\widetilde{R} \in \mathcal{F}}\left( \sum_{\substack{\widetilde{Q} \subseteq \widetilde{R},\ b\tilde{\beta}(\widetilde{Q}) > \epsilon } } \mu(\widetilde{Q}) + \sum_{\substack{\widetilde{Q} \subseteq \widetilde{R},\ \widetilde{I}_\infty(\widetilde{Q}) > \epsilon } } \mu(\widetilde{Q}) \right) &\le \sum_{\widetilde{R} \in \mathcal{F}} \mu(\widetilde{R})\notag\\
& \lesssim \mu(R),\end{align}
where the implicit constant depends on $\epsilon$ $M_\epsilon$, the $d$-regularity constant and dimension and we used the cardinality bound on $\mathcal{F}$, the $d$-regularity of $\widetilde{E}$ and the fact that $\diam(\widetilde{R}) \approx \diam(R)$. Moreover, directly from Lemma \ref{DSlem1.37}
\[\sum_{\substack{Q \subset R \\ I_\infty(Q) > \epsilon}} \mu(Q) \lesssim \mu(R), \]
where the implicit constant depends on $\epsilon$, $d$-regularity and dimension.

Set $C = 3C_u\epsilon$ then by Lemma \ref{geoclaim3.cl} if $\beta(Q) > 3C_b\epsilon$ it must be the case that either $b\tilde{\beta}(\widetilde{Q}) > \epsilon$, $\widetilde{I}_\infty(\widetilde{Q}) > \epsilon$ or $I_\infty(Q) > \epsilon$. Using that
\begin{align*}
   & \int_{R \cap \widetilde{E}} \left(\sum_{\substack{Q \ni x,\ Q\subseteq R,\ \beta(Q) > 3C_b\epsilon}} 1 \right)\, d\mu(x)   \le\sum_{\substack{Q\subseteq R,\ \beta(Q) > 3C_b\epsilon,\  Q \cap \widetilde E \neq \emptyset}} \mu(Q)\notag\\
& \lesssim_L   \sum_{\substack{Q \subset R,\  I_\infty(Q) > \epsilon}} \mu(Q) + \sum_{\widetilde{R} \in \mathcal{F}}\left( \sum_{\substack{\widetilde{Q} \subseteq \widetilde{R},\ b\tilde{\beta}(\widetilde{Q}) > \epsilon } } \mu(\widetilde{Q}) + \sum_{\substack{\widetilde{Q} \subseteq \widetilde{R},\ \widetilde{I}_\infty(\widetilde{Q}) > \epsilon } } \mu(\widetilde{Q}) \right),
\end{align*}
and that as the expressions on the second line is bounded by $\lesssim A_\epsilon \mu(R)$ we can conclude that \begin{align}\label{bwglbpboundfin.eq}
   & \int_{R \cap \widetilde{E}} \left(\sum_{\substack{Q \ni x,\ Q\subseteq R,\ \beta(Q) > 3C_b\epsilon}} 1 \right)\, d\mu(x)   \lesssim A_\epsilon \mu(R),
\end{align}
where $A_\epsilon$ depends on $\epsilon$, $M_\epsilon$, the $d$-regularity constant and dimension. Here we used \eqref{cantbechosentoomany.eq} and $\mu(\widetilde{Q}(Q)) \approx \mu(Q)$.
Thus, using Chebyschev's inequality, if $\alpha(Q)$ defined by
\[\alpha(Q):=\begin{cases}
1 & \text{ if } \beta(Q) > 3C_b\epsilon
\\ 0 & \text{ otherwise},
\end{cases}\]
there exists $N, \eta > 0$ depending on $\theta$ and $A_\epsilon$ such that
\[\mu\left(\left\{x \in R: \sum_{\substack{Q \ni x,\  Q \subseteq R}}\alpha(Q) < N \right\}\right) \ge \eta \mu(R), \quad \forall R \in \mathbb{D}.\]
Indeed, for if
\[F_N = \left\{x \in R \cap \widetilde{E}: \sum_{\substack{Q \ni x,\  Q \subseteq R}}\alpha(Q) \ge  N \right\},\]
then the estimate \eqref{bwglbpboundfin.eq} above gives
\[N\mu(F_N) \le A_\epsilon \mu(R).\]
In particular,  $N$ sufficiently large $\mu(F_N)< (c\theta/2) \mu(R) \le (1/2) \mu(R \cap \widetilde{E})$ and hence
\[\mu(R \setminus F_N) \ge (1/2)\mu(R \cap \widetilde{E}) \ge (c\theta)/2\mu(R) =:\eta \mu(R).\]
Applying Lemma \ref{ds1.12JN.eq} gives Proposition \ref{geolemprop3.prop}.

Now we prove Proposition \ref{geolemprop1.prop}, which is similar to Proposition \ref{geolemprop3.prop}. Let $\mathcal{F}$ be as above. Since $\widetilde{E} \in \GLem (\mathcal{A}, p,q,M)$
\[ \sum_{\widetilde{R} \in \mathcal{F}} \sum_{\widetilde{Q} \subseteq \widetilde{R}} \tilde{\beta}_q(\widetilde{Q})^p\mu(\widetilde{Q}) \lesssim  \sum_{\widetilde{R} \in \mathcal{F}} \mu(\widetilde{R}) \lesssim \mu(R). \]
Again, using Lemma \ref{DSlem1.37} directly we have
\[\sum_{Q \subseteq R} I_q(Q)^p \mu(Q) \lesssim \mu(R).\]
Then using Lemma \ref{geoclaim1.cl}, we obtain
\begin{align*}
    \int_{R \cap \widetilde{E}} \left(\sum_{\substack{Q \ni x,\  Q\subseteq R }} \beta_q(Q)^p \right)\,
    d\mu(x)
    & \le \sum_{\substack{Q \subset R,\  Q \cap \widetilde{E} \neq \emptyset}} \beta_q(Q)^p \mu(Q)
    \\ &\lesssim_{L,p,q} \sum_{\widetilde{R} \in \mathcal{F}} \sum_{\widetilde{Q} \subseteq \widetilde{R}} \tilde{\beta}_q(\widetilde{Q})^p\mu(\widetilde{Q})\\
     &\quad\quad + \sum_{Q \subseteq R} I_q(Q)^p \mu(Q)
    \\& \le A' \mu(Q),
    \end{align*}
    where $A'$ depends on $M$, the $d$-regularity constant and dimension. Arguing along the same lines as above we can conclude that there exist $N' > 0$ and $\eta' > 0$ such that
    \[\mu\left(\left\{x \in R: \sum_{\substack{Q \ni x \\ Q \subseteq R}} \beta_q(Q)^p < N' \right\}\right) \ge \eta' \mu(R), \quad \forall R \in \mathbb{D}.\]
    Applying Lemma \ref{ds1.12JN.eq} gives Proposition \ref{geolemprop1.prop}.

\subsection{Proof of Lemma \ref{ds1.12JN.eq}-Lemma \ref{DSlem1.37}}

\begin{proof}[Proof of Lemma \ref{ds1.12JN.eq}] The lemma is of John-Nirenberg type the lemma holds in our setting with no modifications compared to proof in \cite{DS2}.
\end{proof}

\begin{proof}[Proof of Lemma \ref{geoclaim1.cl}] The proof follows almost exactly as in \cite[Lemma IV.1.20]{DS2}
Fix $Q, \widetilde{Q}, q$ as in the hypotheses of the claim and $\eta>0$. Let $A\in \mathcal{A}$ be such that
\begin{equation}
	\left(\mu(\widetilde{Q})^{-1} \int_{2\widetilde{Q}} \dist(\widetilde{y}, A)^q\, d\mu(\widetilde{y})\right)^{1/q} \leq [\diam \widetilde{Q}] \widetilde{\beta}_q(\widetilde{Q}) + \eta.
\end{equation}
By definition we also have, for this choice of $A$,
\begin{equation}
	[\diam Q]\beta_q(Q) \leq \left( \mu(Q)^{-1} \int_{2Q} \dist(y,A)^q\, d\mu(y)\right)^{1/q}.
\end{equation}
To simplify notation, in what follows we set $\rho(u)=\dist(u,A)$ for $u\in X$. By the triangle inequality (with a constant in the case $0<q<1$) we have
\begin{equation}\label{eq-gl-last}
\begin{split}
	[\diam Q] \beta_q(Q) & \leq C(q)\left(\mu(Q)^{-1} \int_{2Q\cap \widetilde{E}} \rho(y)^q\, d\mu(y)\right)^{1/q} \\
	& \qquad + C(q)\left( \mu(Q)^{-1} \int_{2Q\cap E_1} \rho(y)^q\, d\mu(y)\right)^{1/q}\\
	& \leq C(q, C, d) ([\diam \widetilde{Q}] \widetilde{\beta}_q(\widetilde{Q}) + \eta) \\
	& \qquad + C(q)\left( \mu(Q)^{-1} \int_{2Q\cap E_1} \rho(y)^q\, d\mu(y)\right)^{1/q} ,
\end{split}
\end{equation}
where we have used that $E, \widetilde{E}$ are $d$-regular with constant $C$ to get $\mu(Q) \approx \mu(\widetilde{Q})$, with constants depending only on the homogeneous dimension $d$ and the regularity constant $C$, and also the fact that $2Q\cap \widetilde{E} \subset 2\widetilde{Q}$ by the properties of $\widetilde{Q}$. It thus remains to estimate the last expression in the above.

We define the following ``multiplicity" function $M: \widetilde{E}\times(0,\infty)\to \RR$ given by
\begin{equation}
	M(z,s):= \int_{\substack{ w\in E_1, \, \widetilde{\delta}(w)\leq s\\ \dist(z,w)\leq 2\widetilde{\delta}(w) }} \widetilde{\delta}(w)^{-d}\, d\mu(w)=: \int_{F(z,s)} \widetilde{\delta}(w)^{-d}\, d\mu(w).	
\end{equation}

The first basic property of $M$ that we will need is that there exists $K_0>0$ such that for every $u\in X$ and $s>0$ it holds
\begin{equation}\label{eq-gl-multiplicity-function-bounded}
	\int_{B(u,s)\cap \widetilde{E}} M(z,s)\, d\mu(z) \leq K_0 s^d.
\end{equation}
This is a simple application of Fubini's Theorem:
\begin{equation}
\begin{split}
	\int_{B(u,s)\cap \widetilde{E}} M(z,s)\, d\mu(z) & = \int_{\widetilde{E}} \int_{E} \widetilde{\delta}(w)^{-d}1_{F(z,s)}(w)1_{B(u,s)\cap \widetilde{E}}(z)\, d\mu(w)\, d\mu(z)\\
	& \leq \int_{B(u,3s)\cap E}\widetilde{\delta}(w)^{-d} \int_{B(w,2\widetilde{\delta}(w))\cap \widetilde{E}}\, d\mu(z)\, d\mu(w),
\end{split}
\end{equation}
where we have used the fact that if $w\in F(z,s)$ then $\dist(u,w)\leq \dist(u,z)+ \dist(z,w) \leq s+ 2\widetilde{\delta}(w)\leq 3s$. The desired bound now follows from the regularity of $E$ and $\widetilde{E}$.

We now define, for $K_1>0$ and  the set
\begin{equation}
G(y):= \left\{ z\in \widetilde{E}: z\in B(y,2\widetilde{\delta}(y)), \, M(z,2\widetilde{\delta}(y))\leq K_1 \right\}.
\end{equation}
It follows from \eqref{eq-gl-multiplicity-function-bounded} and Chebyshev's inequality that there exists $K_1$, depending only on $d$ and the regularity constant $C$, such that
\begin{equation}\label{eq-gl-size-G}
	\mu(G(y))\geq c\widetilde{\delta}(y)^d,
\end{equation}
for some constant $c>0$ also depending only on $d$ and $C$. Notice also that the reverse inequality, with a different constant, follows immediately from the regularity of $\widetilde{E}$ and the fact that $G(y)\subset B(y, 2\widetilde{y})$.

We claim that, for every $y\in E_1\cap 2Q$,
\begin{equation}\label{rhoqest.eq}
	\rho(y)^q\leq C(q,d,C) \widetilde{\delta}(y)^q + C(q,d,C) \widetilde{\delta}(y)^{-d}\int_{G(y)} \rho(z)^q\, d\mu(z).
\end{equation}
To see this we fix $y$ as above and $z\in G(y)$ to obtain, by the triangle inequality,
\begin{equation}
	\rho(y)^q \leq C(q)\dist(y,z)^q + C(q)\rho(z)^q.
\end{equation}
Integrating the $z$ variable over $G(y)$, and using \eqref{eq-gl-size-G}, the estimate \eqref{rhoqest.eq} follows. We use this to estimate
\begin{align}\label{eq-gl-semifinal}
& \mu(Q)^{-1} \int_{2Q\cap E_1} \rho(y)^q\, d\mu(y)\notag\\
 & \leq C(q,d,C)\mu(Q)^{-1} \int_{2Q\cap E_1} \widetilde{\delta}(y)^q\, d\mu(y)\notag \\
	 & \quad + C(q,d,C)\mu(Q)^{-1}\int_{2Q\cap E_1} \widetilde{\delta}(y)^{-d}\int_{G(y)} \rho(z)^q\, d\mu(z)\, d\mu(y)\notag\\
	 & \leq C(q,d,C)\mu(Q)^{-1}\int_{2Q\cap E_1} \widetilde{\delta}(y)^{-d}\int_{G(y)} \rho(z)^q\, d\mu(z)\, d\mu(y)\notag\\
	 & \quad + C(q,d,C)[\diam Q]^qI_q(Q)^q\notag\\
	 & =: C(q,d,C)[\diam Q]^q( J_q(Q)^q + I_q(Q)^q).
\end{align}
Using Fubini's Theorem, together with the fact that for $y\in 2Q$ we have $G(y)\subset B(y,2\widetilde{\delta}(y))\cap\widetilde{E} \subset 2\widetilde{Q}$ by definition of $G(y)$ and $\widetilde{Q}$, we can estimate $J_q$ as follows
\begin{align}\label{eq-gl-jq-estimate}
	&[\diam Q]^q J_q(Q)^q \notag\\
& = \mu(Q)^{-1} \int_{\widetilde{E}} \int_E \widetilde{\delta}(y)^{-d}\rho(z)^q 1_{G(y)}(z) 1_{2Q\cap E_1}(y)\, d\mu(z)\, d\mu(y)\\
	& \leq \mu(Q)^{-1} \int_{2\widetilde{Q}} \int_{\substack{ 2Q\cap E_1\\ z\in G(y)}} \widetilde{\delta}(y)^{-d} \, d\mu(y) \rho(z)^q\, d\mu(z).
\end{align}
We now claim that the inner integral is bounded by the constant $K_1$, i.e.
\begin{equation}\label{eq-gl-counting-final}
	\int_{\substack{ 2Q\cap E_1\\ z\in G(y)}} \widetilde{\delta}(y)^{-d} \, d\mu(y)\leq K_1, \qquad \forall z\in \widetilde{E}.
\end{equation}
To prove this we fix $z\in \widetilde{E} $ and choose $y_0\in E_1\cap 2Q$ such that $z\in G(y_0)$ (if no such $y_0$ exists then the integral is zero and we're done), with the additional property
\begin{equation}
	\widetilde{\delta}(y_0) \geq \dfrac{1}{2} \sup\{ \widetilde{\delta}(y): y\in 2Q\cap E_1, \, z\in G(y) \}.
\end{equation}
By definition of $G(y_0)$ we have $M(z,2\widetilde{\delta}(y_0))\leq K_1$, i.e.
\begin{equation}
	\int_{F(z,2\widetilde{\delta}(y_0))} \widetilde{\delta}(y)^{-d}\, d\mu(y) \leq K_1.
\end{equation}
The claim now follows from noting that
\begin{equation}
\{ y \in E: y\in 2Q\cap E_1\, z\in G(y)\} \subset F(z,2\widetilde{\delta}(y_0)).
\end{equation}
This in turn follows from the fact that $z\in G(y)$ implies $|z-y|\leq 2\widetilde{\delta}(y)$, while $\widetilde{\delta}(y)\leq 2\widetilde{\delta}(y_0)$ by our choice of $y_0$. This proves \eqref{eq-gl-counting-final}.

Using \eqref{eq-gl-counting-final} in the estimate for $J_q$ we arrive at
\begin{align}
	[\diam Q] J_q(Q) &\leq C(q,d,C)\left( \mu(Q)^{-1} \int_{2\widetilde{Q}} \rho(z)^q\, d\mu(z) \right)^{1/q}\notag\\
&\leq [\diam\widetilde{Q}]\widetilde{\beta}_q(\widetilde{Q}) + \eta,
\end{align}
where we used our choice of $A\in \mathcal{A}$ for the last inequality.

Plugging this estimate into \eqref{eq-gl-semifinal} we see
\begin{align}
	\mu(Q)^{-1}\int_{2Q\cap E_1} \rho(y)^q\, d\mu(y) &\leq C(q,d,C)[\diam Q]^q(\widetilde{\beta}_q(\widetilde{Q})^q + I_q(Q)^q)\notag\\
& + C(q,d,C)\eta^q.
\end{align}
Going back to \eqref{eq-gl-last} and letting $\eta\to 0$ the result follows.
\end{proof}

\begin{proof}[Proof of Lemma \ref{geoclaim2.cl}]
The proof of this claim  will be omitted. It follows the same lines as Lemma \ref{geoclaim3.cl}, and is in fact simpler. The idea is to mimic the argument in the proof of that claim, with the obvious modifications, up to the estimate \eqref{eq-bwgl-inequality} at which point we let $\eta\to 0$.
\end{proof}

\begin{proof}[Proof of Lemma \ref{geoclaim3.cl}]
Fix $Q, \widetilde{Q}$ as in the statement of the claim and $\eta>0$. Let $A\in \mathcal{A}$ be such that
\begin{equation}\label{eq-bwgl-eta-minimizer}
	\sup_{y\in 2\widetilde{Q}} \dist(y, A) + \sup_{z\in A\cap B(x_{\widetilde{Q}},2\diam\widetilde{Q})} \dist(z,\widetilde{E}) \leq [\diam \widetilde{Q}]b\widetilde{\beta}(\widetilde{Q}) + \eta.
\end{equation}
By definition of $\beta(Q)$, for this choice of $A$ it holds
\begin{equation}\label{eq-bwgl-beta}
\begin{split}
	[\diam Q] b\beta(Q) & \leq \sup_{y\in 2Q} \dist(y,A) + \sup_{z\in A\cap B(x_Q, 2\diam Q)} \dist(z,E)\\
	& =: II + III.
\end{split}
\end{equation}

To estimate $II$ we proceed as follows. Fix $y\in 2Q$ and let $\widetilde{y}\in \widetilde{E}$ such that $\dist(y,\widetilde{y})\leq \widetilde{\delta}(y) + \eta$. By our choice of $\widetilde{Q}$, in particular since $Q\cap \widetilde{E}\neq \emptyset$ and $\diam\widetilde{Q} \geq 10 \diam Q$, we may assume $\widetilde{y}\in 2\widetilde{Q}$ so that
\begin{equation}
\begin{split}
	\dist(y,A) & \leq \dist(y,\widetilde{y}) + \dist(\widetilde{y}, A) \\
	& \leq \widetilde{\delta}(y) + \sup_{\widetilde{y} \in 2\widetilde{Q}} \dist(\widetilde{y}, A) + \eta.
\end{split}
\end{equation}
Taking the supremum over all $y\in 2Q$ we arrive at
\begin{equation}\label{eq-bwgl-inequality}
\begin{split}
	II & = \sup_{y\in 2Q} \dist(y,A)  \leq \sup_{y\in 2Q} \widetilde{\delta}(y) + \sup_{\widetilde{y}\in w\widetilde{Q}} \dist(\widetilde{y}, A) + \eta \\
	& = \sup_{\substack{y\in 2Q\cap E_1\\ \widetilde{\delta}(y)\leq 2\diam Q}} \widetilde{\delta}(y) + \sup_{\widetilde{y}\in 2\widetilde{Q}} \dist(\widetilde{y},A) + \eta\\
	& = [\diam Q]I(Q) + \sup_{\widetilde{y}\in 2\widetilde{Q}} \dist(\widetilde{y}, A) + \eta,
\end{split}
\end{equation}
where in the second line we have used the definition of $\widetilde{\delta}= \dist(\cdot, \widetilde{E})$, and again the fact that $Q\cap \widetilde{E}\neq \emptyset$.

To estimate $III$ we proceed similarly. Fix $z\in A\cap B(x_Q, 2\diam Q)$ and let $\widetilde{y}\in \widetilde{E}$ be such that $\dist(y,\widetilde{y}) \leq d(z, \widetilde{E}) + \eta$. Arguing as before we see that we can choose $\widetilde{y}$ such that $\widetilde{y}\in 2\widetilde{Q}$ and $\delta(\widetilde{y})\leq 2\diam \widetilde{Q}$, so that
\begin{equation}
\begin{split}
	\dist(z,E) & \leq \dist(z,\widetilde{y}) + \dist(\widetilde{y}, E)\\
	& \leq \dist(z,\widetilde{E}) + \dist(\widetilde{y}, E) + \eta\\
	& \leq \dist(z,\widetilde{E}) + \sup_{\substack{\widetilde{y}\in 2\widetilde{Q}\\ \delta(\widetilde{y})\leq 2\diam \widetilde{Q}}} \dist(\widetilde{y},E) + \eta\\
	& = \dist(z,\widetilde{E}) + [\diam \widetilde{Q}]\widetilde{I}(\widetilde{Q}) + \eta.
\end{split}
\end{equation}
Taking the supremum over all such $z$ gives
\begin{equation}
	III\leq \sup_{z\in A\cap B(x_{\widetilde{Q}}, 2\diam \widetilde{Q})} \dist(z, \widetilde{E}) + [\diam \widetilde{Q}]\widetilde{I}(\widetilde{Q}) + \eta.
\end{equation}
Combining the estimates for $II$ and $III$ we arrive at
\begin{equation}
\begin{split}
	II+III & \leq \sup_{\widetilde{y}\in 2\widetilde{Q}} \dist(\widetilde{y}, A ) + \sup_{z\in A\cap B(x_{\widetilde{Q}}, 2\diam \widetilde{Q})} \dist(z, \widetilde{E}) \\
	& \qquad + [\diam Q]I(Q) + [\diam \widetilde{Q}]\widetilde{I}(\widetilde{Q}) + 2\eta\\
	& \leq [\diam \widetilde{Q}] b\widetilde{\beta}(\widetilde{Q})  +[\diam Q]I(Q) + [\diam \widetilde{Q}]\widetilde{I}(\widetilde{Q}) + 3\eta,
\end{split}
\end{equation}
where we used \eqref{eq-bwgl-eta-minimizer} for the last line. Plugging this last estimate into \eqref{eq-bwgl-beta} and letting $\eta \to 0$ the result follows.
\end{proof}

\begin{proof}[Proof of Lemma \ref{DSlem1.37}]
First note that the last statement follows from the case $q= \infty$ $p = 2d$ and Chebyshev's inequality. In the following we give the proof in the case $p = q$ (this case always satisfies the inequality for $d,p$ and $q$). Using Tonelli's theorem
\begin{align*}
    \sum_{Q \subseteq R} I_q(Q)^q\mu(Q) & = \sum_{Q \subseteq R} I_q(Q)^p\mu(Q)\\
     &= \sum_{Q \subseteq R}\int\limits_{\substack{z \in 2Q \\ \dist(z, E_3) < 2\diam(Q)}} [\dist(z, E_3)(\diam(Q))^{-1}]^{q} \, d\mu(z)
    \\& = \int_{2R} \left(\sum_{\substack{2Q \ni z \\ 2\diam(Q) \ge \dist(z, E_3)}} [\dist(z, E_3)(\diam(Q))^{-1}]^{q}  \right) \, d\mu(z)
    \\ & = C_q \int_{2R \cap E_3} 1 \, d\mu(z) \le C \mu(R).
\end{align*}
Now notice that $I_q(Q) < C_q$, so that the case $q< p$ easily reduces to the case $p = q$. The question then becomes how large we can make $q$. Observe that
\[I_\infty(Q)^{(r+ d)/r} \le C_r I_r(\widehat Q),\]
where $\widehat{Q}$ is the smallest cube containing $Q$ for which $\diam(\widehat{Q}) > 2 \diam(Q)$. This follows from the fact that if $z \in 2Q$ with $0 < \dist(z, E_3) < 2\diam(Q)$ then by the $d$-regularity of $E$
\begin{align*}&[\dist(z,E_3) \diam(Q)^{-1}]^{r + d}\\
 &\le C\mu(Q)^{-1}  \int\limits_{B(z, \dist(y,E_3)/10) \cap E_2}[\dist(w, E_3)\diam(Q)^{-1}]^r \, d\mu(w) \\& \le CI_r(\widehat Q)^r.\end{align*}
Notice that we can make the assumption that $\dist(z, E_3) > 0$ since the $z$ for which $\dist(z,E_3) = 0$ do not factor into the definition of $I_\infty(Q)$ unless $I_\infty(Q) = 0$. The rest of the proof is just playing `the exponent game' and we refer the reader to \cite{DS2} for the details. \end{proof}

\end{document}